\documentclass[11pt]{amsart}

\usepackage[margin=1.15in]{geometry}
\usepackage{microtype}

\usepackage{amsmath,amssymb,amsthm}
\usepackage{mathtools}
\usepackage{bm}
\usepackage[pagewise]{lineno}
\usepackage{enumitem}
\usepackage[colorlinks=true,linkcolor=blue,citecolor=blue,urlcolor=blue]{hyperref}
\numberwithin{equation}{section}

\theoremstyle{plain}
\newtheorem{theorem}{Theorem}[section]
\newtheorem{lemma}[theorem]{Lemma}
\newtheorem{proposition}[theorem]{Proposition}

\theoremstyle{definition}
\newtheorem{definition}[theorem]{Definition}

\theoremstyle{remark}
\newtheorem{remark}[theorem]{Remark}

\title{Unique mild solution for the ES-BGK model with the correct Prandtl number}

\author{Hyun-Woo Jeong}
\thanks{Hyun-Woo Jeong:
Department of Mathematics, Sungkyunkwan University,
Suwon 16419, Republic of Korea.
Email: \texttt{hwj1222@skku.edu}.}

\author{Sungsu Park}
\thanks{Sungsu Park:
Department of Mathematics, Sungkyunkwan University,
Suwon 16419, Republic of Korea.
Email: \texttt{ppssu@skku.edu}.
Corresponding author.}

\date{}

\begin{document}
	
	\begin{abstract}
        We prove the existence and uniqueness of mild solutions to the ellipsoidal BGK (ES-BGK) model with the correct Prandtl number, under suitable weighted bounds and a free-transport non-vacuum condition. At the correct Prandtl number, the temperature tensor is no longer uniformly comparable to the scalar temperature, so classical scalar moment estimates do not directly yield the required Gaussian bounds. We derive anisotropic velocity moment estimates by splitting velocity space along the longest principal direction of the stress tensor and using the geometry of transverse elliptical slices. These estimates yield a linear weighted $L^{\infty}$ bound for the ellipsoidal Gaussian. Together with uniform ellipticity and weighted $L^1$-Lipschitz continuity under the stated assumptions, this bound allows us to construct unique mild solutions.  
	\end{abstract}
	
	\maketitle
	\tableofcontents
   \noindent\textbf{keywords:} ES-BGK model; global well-posedness; anisotropic Gaussian; weighted $L^\infty$ estimate.
	\section{Introduction}
	\label{sec:introduction}
    
    We consider the Cauchy problem for the ellipsoidal BGK (ES-BGK) model, which is a relaxation model of the Boltzmann equation
    \begin{align}\label{Eq:ESBGK nu}
        \begin{split}
            &\partial_t f +v\cdot \nabla_x f=\frac{1}{1-\nu}(\mathcal{M}_\nu[f]-f),\cr
            &f(0)=f_0.
        \end{split}
    \end{align} 
    The velocity distribution $f(t,x,v)$ denotes the number density of gas particles in the phase space $\mathbb{R}^3\times\mathbb{R}^3$ at time $t\ge 0$. The non-isotropic Gaussian $\mathcal{M}_{\nu}[f]$ is defined by
    \begin{align*}
        \begin{split}
            \mathcal{M}_{\nu}[f]=\frac{\rho}{\sqrt{\det(2\pi \mathcal{T}_{\nu})}}\exp\left(-\frac{1}{2}(v-U)^T\mathcal{T}^{-1}_{\nu}(v-U) \right),
        \end{split}
    \end{align*}
    where $\rho$, $U$, $T$, and $\Theta$ are the macroscopic density, bulk velocity, temperature, and stress tensor, respectively, which are given by
    \begin{align*}
        \begin{split}
            \rho(t,x)&=\int_{\mathbb{R}^3}f(t,x,v)dv,\cr
            \rho(t,x)U(t,x)&=\int_{\mathbb{R}^3}vf(t,x,v)dv,\cr
            3\rho(t,x)T(t,x)&=\int_{\mathbb{R}^3} |v-U(t,x)|^2f(t,x,v)dv,\cr
            \rho(t,x)\Theta(t,x)&=\int_{\mathbb{R}^3}(v-U)\otimes (v-U) f(t,x,v)dv\cr
            &=\int_{\mathbb{R}^3}v\otimes v f(t,x,v)dv-\rho U\otimes U.
        \end{split}
    \end{align*}
    The temperature tensor $\mathcal{T}_{\nu}$ is defined as a linear combination of $TI$ and $\Theta$:
    \begin{align*}
        \begin{split}
            \mathcal{T}_{\nu}&=
            \begin{pmatrix}
                (1-\nu)T+\nu\Theta_{11} & \nu\Theta_{12} & \nu\Theta_{13}\\
                \nu\Theta_{21} & (1-\nu)T+\nu\Theta_{22} & \nu\Theta_{23}\\
                \nu\Theta_{31} & \nu\Theta_{32} & (1-\nu)T+\nu\Theta_{33}
            \end{pmatrix}\cr
            &=(1-\nu)T I+\nu\Theta.
        \end{split}
    \end{align*}
    The relaxation operator $\mathcal{M}_\nu[f]-f$ shares important properties with the Boltzmann equation. One is the conservation laws of mass, momentum, and energy:
    \begin{align*}
        \begin{split}
            \int_{\mathbb{R}^6}(\mathcal{M}_{\nu}[f]-f)\begin{pmatrix}
				1\\ v\\ |v|^2
			\end{pmatrix}dxdv=0.
        \end{split}
    \end{align*}
    The other is entropy dissipation, which was verified in \cite{andries2000gaussian}:
    \begin{align*}
        \begin{split}
            \frac{d}{dt}\int_{\mathbb{R}^6} f\ln fdxdv\leq 0.
        \end{split}
    \end{align*}

    The Boltzmann equation provides a fundamental kinetic description of dilute gases, but the complexity of its collision operator makes numerical simulations computationally demanding. To reduce this cost, the BGK model was introduced in \cite{bhatnagar1954model} by replacing the collision operator with a relaxation term. Its computational simplicity and preservation of key physical properties of the Boltzmann equation have led to its widespread use in science and engineering. However, although the BGK model recovers the correct Euler limit, its Navier--Stokes approximation yields an incorrect Prandtl number, which is proportional to the ratio between the viscosity and the thermal conductivity. To address this limitation, Holway \cite{Holway1966} introduced the ES-BGK model, whose anisotropic Gaussian relaxation state allows the Prandtl number to be adjusted. 

    \subsection{The correct Prandtl number at
\texorpdfstring{$\nu=-1/2$}{nu=-1/2}} The dimensionless parameter \(\nu\in[-1/2,1)\), called the Prandtl parameter, determines how stress anisotropy enters the relaxation Gaussian. Indeed, the temperature tensor can be written as 
    \begin{align*}
        \begin{split}
            \mathcal T_\nu=TI+\nu(\Theta-TI),   
        \end{split}
    \end{align*}
    where \(\Theta-TI\) is the traceless part of the stress tensor. In particular, \(\nu=0\) gives an isotropic Maxwellian and recovers the classical BGK model. The parameter \(\nu\) adjusts the relative relaxation rates of viscous stress and heat flux, and the Chapman–Enskog expansion yields
    \begin{align*}
        \begin{split}
            \mathrm{Pr}=\frac{c_p\mu}{\kappa}=\frac{1}{1-\nu},    
        \end{split}
    \end{align*}
    where \(\mu\), \(\kappa\), and \(c_p\) denote the viscosity, thermal conductivity, and specific heat at constant pressure, respectively. Thus, the classical BGK model predicts \(\mathrm{Pr}=1\), whereas the critical value \(\nu=-1/2\) gives \(\mathrm{Pr}=2/3\), the value conventionally adopted for dilute monatomic gases. Throughout this paper, we refer to $\mathrm{Pr}=2/3$ as the correct Prandtl number.

Introducing anisotropy into the relaxation Gaussian creates additional analytical difficulties. In the classical BGK model, the covariance matrix of the relaxation Maxwellian is $TI$, so estimates concerning the thermal structure reduce to controlling the scalar temperature \(T\). In the ES-BGK model, however, the variances depend on direction, and estimating the determinant and inverse of \(\mathcal T_\nu\) requires information about its eigenvalues. The identity $\operatorname{tr}\mathcal T_\nu=3T$ controls only the sum of the eigenvalues and, by itself, does not
yield a bound of the form $\mathcal T_\nu\ge cTI$ with a uniform constant $c>0$.

For the noncritical range \(-1/2<\nu<1\), this difficulty can be handled through the comparison \cite{yun2015classical,yun2015ellipsoidal}:
\begin{align*}
    \begin{split}
        c_\nu TI\le\mathcal T_\nu\le C_\nu TI,
\qquad
c_\nu=\min\{1-\nu,1+2\nu\}>0,
\qquad
C_\nu=\max\{1-\nu,1+2\nu\}.   
    \end{split}
\end{align*}
This reduces estimates for the anisotropic Gaussian to derives involving the scalar temperature. In the classical BGK setting, Perthame and Pulvirenti established Gaussian height, scalar second-moment, and mean-velocity estimates \cite{perthame1993weighted}
\begin{align*}
		\frac{\rho}{T^{3/2}},
		\qquad
		\rho\bigl(T+|U|^2\bigr)^{(q-3)/2},
		\qquad
		\frac{\rho|U|^{q+3}}
		{T^{3/2}\bigl(T+|U|^2\bigr)^{3/2}}.
	\end{align*}
Together with the tensor comparison above, these estimates yield weighted \(L^\infty\) bounds for the ES-BGK Maxwellian, with constants depending on \(\nu\).

At the critical point \(\nu=-1/2\), however, the comparison constant \(c_\nu\) vanishes. The temperature tensor takes the form
\begin{align*}
    \begin{split}
        \mathcal T_{-1/2}
=\frac12\bigl((\operatorname{tr}\Theta)I-\Theta\bigr).    
    \end{split}
\end{align*}
Although this tensor remains positive definite for every nonzero nonnegative \(L^1_2\) density, its smallest eigenvalue need not be comparable to \(T\). For example, if
\begin{align*}
    \begin{split}
        \Theta=\operatorname{diag}(1,\delta,\delta),
\qquad 0<\delta<1,   
    \end{split}
\end{align*}
then
\begin{align*}
    \begin{split}
        T=\frac{1+2\delta}{3},
\qquad
\mathcal T_{-1/2}
=\operatorname{diag}\left(\delta,\frac{1+\delta}{2},
\frac{1+\delta}{2}\right).    
    \end{split}
\end{align*}
Consequently, the smallest eigenvalue tends to zero as \(\delta\to0\), while the scalar temperature remains bounded away from zero. Thus, the classical temperature-comparison argument does not extend directly to this critical case.

This breakdown of temperature comparison arguments at $\nu=-1/2$ demands additional analytical estimates in the global construction of mild solutions. Motivated by this difficulty, we develop anisotropic velocity moment estimates that directly account for the geometry of the stress tensor. These estimates yield a linear weighted \(L^\infty\) bound for the critical relaxation Gaussian and provide the basis for constructing unique mild solutions under suitable weighted bounds and a free-transport non-vacuum condition. More precisely, the first three estimates in Proposition~\ref{prop:four-moment-estimates} are ellipsoidal counterparts of the Perthame--Pulvirenti estimates, controlling the Gaussian height scale, the scalar second moment, and the mean velocity. However, they do not directly provide the coupled control of the determinant and the largest principal variance required for the weighted Gaussian bound. In the isotropic case, the same scalar \(T\) determines both the Gaussian volume and its longest length scale. These quantities are distinct for a general stress tensor. Writing its eigenvalues as \(\theta_1\ge\theta_2\ge\theta_3>0\), we derive the additional directional estimate (see \eqref{item:B4} in Proposition~\ref{prop:four-moment-estimates})
\begin{align*}
    \begin{split}
        \frac{\rho\theta_1^{q/2}}{\sqrt{\det\Theta}}
\le C_q\|f\|_{L^\infty_q},
\qquad q>5.    
    \end{split}
\end{align*}
This controls the product of the natural height scale \(\rho/\sqrt{\det\Theta}\) and the \(q\)-th power of the longest principal length scale \(\sqrt{\theta_1}\). Finally, the relations
\begin{align*}
    \det\mathcal T_{-1/2}\ge\det\Theta,
\qquad
\lambda_{\max}(\mathcal T_{-1/2})\le\theta_1    
\end{align*}
allow us to combine the directional estimate with the determinant and mean-velocity estimates to obtain
\begin{align*}
    \operatorname*{ess\,sup}_{v\in\mathbb{R}^3}|(1+|v|^2)^{\frac{q}{2}}\mathcal M_{-1/2}[f](t,x,v)| \le C_q \operatorname*{ess\,sup}_{v\in\mathbb{R}^3}|(1+|v|^2)^{\frac{q}{2}}f(t,x,v)|, \qquad q>5,    
\end{align*}
with a constant independent of the eigenvalue ratios of \(\Theta\). This estimate replaces the use of scalar temperature comparison in controlling the weighted Maxwellian and is the main analytical ingredient in our treatment of the critical case.

\subsection{Main result} 
Before introducing our main result, we record the following notation which will be used throughout the paper.
	\begin{itemize}
		\item The angle bracket is denoted by
        \begin{align*}
            \begin{split}
                \langle a\rangle:=(1+|a|^2)^{1/2}.    
            \end{split}
        \end{align*}

         \item For functions depending on $(x,v)\in\mathbb{R}^3\times\mathbb{R}^3$, we define for $q,r\ge 0$,
        \begin{align*}
        \begin{split}
            \|h(t)\|_{L^{\infty}_q}&=\operatorname*{ess\,\sup}_{(x,v)\in\mathbb{R}^3\times\mathbb{R}^3} \langle v\rangle^q|h(t,x,v)|,\cr
            \|h(t)\|_{L^{\infty}_{r,q}}&=\operatorname*{ess\,\sup}_{(x,v)\in\mathbb{R}^3\times\mathbb{R}^3}\langle x\rangle^r \langle v\rangle^q|h(t,x,v)|,\cr
            \|h(t)\|_{L^{1}_{2}}&=\int_{\mathbb{R}^6}\langle v\rangle^2|h(t,x,v)|dxdv.
        \end{split}
    \end{align*}
	
		\item We use $\rho$, $U$, $T$, $\Theta$, $\mathcal{T}_{-1/2}$ and $\rho_f$, $U_f$, $T_f$, $\Theta_f$, $\mathcal{T}_{-1/2}[f]$ interchangeably to denote the macroscopic fields constructed from $f$. We omit the subscript $f$ when there is no risk of confusion.

        \item For matrices, $|A|$ denotes the Frobenius norm. For symmetric matrices
		$A$ and $B$, we write $A\leq B$ if $B-A$ is positive semidefinite. We denote the largest and smallest
		eigenvalues of a symmetric matrix $A$ by
        \begin{align*}
            \begin{split}
                \lambda_{\max}(A),
		\qquad
		\lambda_{\min}(A).   
            \end{split}
        \end{align*}
		
		\item The eigenvalues of $\Theta_f$ are ordered as
        \begin{align*}
            \begin{split}
                \theta_1\ge\theta_2\ge\theta_3>0.   
            \end{split}
        \end{align*}
		Thus, we have
        \begin{align*}
            \begin{split}
                \lambda_{\max}(\Theta_f)=\theta_1,
		\qquad
		\sqrt{\det\Theta_f}
		=
		\sqrt{\theta_1}\sqrt{\theta_2}\sqrt{\theta_3}.   
            \end{split}
        \end{align*}
		
		\item $C_{a,b,c,\cdots}$ means a constant which depends on $a$, $b$, $c$, $\cdots$. Moreover, these constants may take different values in each line. If no subscript is specified, like $C$, this represents a generic constant.
	\end{itemize}
    
    For the analysis at $\nu=-1/2$, we normalize the constant collision
    frequency to one by rescaling time and space, and henceforth consider
\begin{align}
    \partial_t f+v\cdot\nabla_x f
    =\mathcal M_{-1/2}[f]-f,
    \qquad f|_{t=0}=f_0.
    \label{eq:endpoint-esbgk-eq}
\end{align}

Now we are ready to state our main result.
\begin{definition}
    For $T_*>0$, we say $f\in C([0,T_*];L^1_2(\mathbb{R}^6))$ is a mild solution to \eqref{eq:endpoint-esbgk-eq} if $f$ satisfies 
    \begin{align*}
        \begin{split}
            f(t,x,v)&=e^{-t}f_0(x-vt,v)+\int^t_0e^{-(t-s)}\mathcal{M}_{-1/2}[f](s,x-v(t-s),v)ds.
        \end{split}
    \end{align*}
\end{definition}

\begin{theorem}
		\label{thm:endpoint-mild-solution}
		Let $r>3$, $q>5$, and $T_*>0$. Assume that $f_0\ge0$ satisfies
		\begin{align*}
			\|f_0\|_{L^\infty_{q+r}}+\|f_0\|_{L^{\infty}_{r,q}}<\infty,
			\qquad
			\|f_0\|_{L^1_2}<\infty, \qquad \int_{\mathbb{R}^6}f_0|\ln f_0|dxdv<\infty,
		\end{align*}
		and that there exists $\rho_0>0$ such that
		\begin{align}
			\int_{\mathbb{R}^3}f_0(x-vt,v)dv\ge \rho_0\langle x\rangle^{-r},
			\label{eq:free-transport-nonvacuum-endpoint}
		\end{align}
        for every $t\in[0,T_*]$ and almost every $x \in\mathbb{R}^3$. Then \eqref{eq:endpoint-esbgk-eq} admits a unique nonnegative mild solution
		\begin{align*}
			f\in C\bigl([0,T_*];L^1_2(\mathbb R^3\times\mathbb R^3)\bigr),
		\end{align*}
		satisfying
		\begin{align}
			\|f(t)\|_{L^\infty_{q+r}}
			\le
			e^{C_1t}\|f_0\|_{L^\infty_{q+r}},\quad \|f(t)\|_{L^\infty_{r,q}}\le C_2e^{C_3t}(\|f_0\|_{L^\infty_{r,q}}+\|f_0\|_{L^{\infty}_{q+r}}),
			\quad 0\le t\le T_*,
			\label{eq:endpoint-solution-Linfty-bound}
		\end{align}
        where $C_i\;(i=1,2,3)$ depends on $q$, $r$, $T_*$, and $f_0$, and
		\begin{align}
			\rho_f(t,x)
			\ge
			e^{-T_*}\rho_0\langle x\rangle^{-r},
			\qquad
			\text{for every }t\in[0,T_*]\text{ and almost every }x\in\mathbb{R}^3.
			\label{eq:endpoint-solution-density-lower}
		\end{align}
		Moreover, there exist constants $0<c_l<c_u$ depending on
		$q,r,T_*,\rho_0$ and $f_0$ such that
		\begin{align}
			c_lI
			\le
			\mathcal T_{-1/2}(t,x)
			\le
			c_uI,
			\qquad
			\text{for every }t\in[0,T_*]\text{ and almost every }x\in\mathbb{R}^3.
			\label{eq:endpoint-solution-ellipticity}
		\end{align}
		Finally, the conservation laws and the entropy dissipation hold:
		\begin{align}
			\frac{d}{dt}
			{\int_{\mathbb R^3}\int_{\mathbb R^3}}
			f(t,x,v)
			\begin{pmatrix}
				1\\ v\\ |v|^2
			\end{pmatrix}
			\,dv\,dx
			=0,\qquad 0\leq t\leq T_*,
			\label{eq:endpoint-conservation-laws}
		\end{align}
        and
        \begin{align}\label{eq:endpoint-entropy-dissipation}
            \begin{split}
                \int_{\mathbb{R}^6}f(t)\ln f(t)dxdv\leq \int_{\mathbb{R}^6}f_0\ln f_0dxdv,\qquad 0\leq t\leq T_*.
            \end{split}
        \end{align}
	\end{theorem}

    \begin{remark}
        Theorem~\ref{thm:endpoint-mild-solution} also holds on the spatial torus $x\in\mathbb{T}^3$, with the spatial weights omitted.    
    \end{remark}

   \subsection{Literature review}
   We briefly review the mathematical results on the BGK and ES-BGK models. For the classical BGK model, global existence of weak solutions under finite mass, energy, and entropy assumptions was proved in \cite{perthame1989global}. Weighted $L^{\infty}$ estimates and uniqueness were obtained in \cite{perthame1993weighted}, and the corresponding whole-space theory was developed in \cite{mischler1996uniqueness}. Extensions to $L^p$ spaces were investigated in \cite{zhang2007lp}. Stationary solutions on a finite interval with large inflow boundary data were constructed in \cite{ukai1992stationary}, while global existence and asymptotic behavior near a global Maxwellian were studied in \cite{bellomo1985cauchy,yun2010cauchy}. Extensions of BGK models have also been studied, including external confining potentials \cite{bosi2009bgk}, Fokker--Planck dynamics \cite{zhang2008global}, and Vlasov--Poisson coupling \cite{zhang2010cauchy}. The various results of the hydrodynamic limit \cite{berthelin2005kinetic,saint2002discrete,saint2002modele,saint2003bgk} and diffusion limit \cite{degond2000diffusion,dolbeault2007non,goudon2001approximation,mellet2010fractional,mellet2011fractional} were investigated. Motivated by the computational efficiency of the BGK model, numerous numerical methods have been developed and analyzed, for example, \cite{issautier1996convergence,mieussens2001,pieraccini2007implicit,russo2009semilagrangian,russo2012convergence,struchtrup1997bgk}.
   
    We now turn to the results concerning the ES-BGK model in the noncritical range of the Prandtl parameter, $-1/2<\nu<1$. The entropy dissipation was established in \cite{andries2000gaussian}, and an alternative derivation of the model together with another proof of the $H$-theorem was presented in \cite{brull2008new}. Quantitative entropy production estimates were obtained in \cite{yun2016entropy}. For fixed collision frequency, global classical solutions in weighted $L^{\infty}$ Sobolev spaces were constructed in \cite{yun2015classical}, and global weak solutions in the whole space were obtained in \cite{park2016cauchy}. Classical solutions near equilibrium were studied in the periodic and whole-space settings in \cite{yun2015ellipsoidal,hwang2019ellipsoidal}, respectively. Global mild solutions for initial data allowing large spatial oscillations were constructed in \cite{duan2017global}. In \cite{bang2016stationary}, stationary solutions in a slab with inflow boundary conditions were studied.

    At the critical value $\nu=-1/2$, which leads to the correct Prandtl number, the lower bound comparison between the temperature tensor and the scalar temperature degenerates, requiring an additional argument to control the anisotropic Gaussian. An entropy production estimate covering the full range $-1/2\leq \nu<1$ was obtained in \cite{kim2022entropy}, and the existence of weak solutions to the Cauchy problem at the critical value was proven in \cite{son2022cauchy}. In \cite{brull2024stationary}, a unique stationary solution in a slab with mixed boundary conditions was constructed. The local well-posedness of the mild solutions was established in \cite{son2025local}. Numerical simulation and analysis for the ES-BGK model were developed in \cite{boscarino2026asymptotic,filbet2011asymptotic,gallis2011investigation,mieussens2004numerical,meng2013assessment,russo2018convergence,zheng2005ellipsoidal}. The present work develops anisotropic velocity moment estimates yielding a linear weighted $L^\infty$ bound for the critical ES-BGK Gaussian. Together with weighted $L^1$ Lipschitz continuity, these estimates allow the construction of unique mild solutions under a free-transport non-vacuum condition without the additional momentum smallness assumption.
    \newline
    \newline

The paper is organized as follows. In Section~\ref{sec:moment-estimates}, we derive ellipsoidal velocity moment estimates and obtain the weighted $L^\infty$ bound for the critical ES-BGK Maxwellian. Section~\ref{sec:ellipticity-lipschitz} establishes uniform ellipticity of the temperature tensor and the $L^1_2$-Lipschitz continuity of the Maxwellian under suitable bounds on the distribution function. In Section~\ref{sec:endpoint-mild-solutions}, we use these estimates to construct a unique mild solution under a free-transport non-vacuum condition and prove conservation laws and entropy dissipation. Finally, Section~\ref{sec:endpoint-weak-convergence} uses these mild solutions as approximations to construct global weak solutions through uniform moment and entropy bounds, velocity averaging, and identification of the nonlinear limit.
	
	\section{Ellipsoidal moment estimates}
	\label{sec:moment-estimates}
    In this section, we establish the weighted $L^{\infty}$ bound for the critical ES-BGK Gaussian for $q>5$. We first verify positive definiteness of the stress tensor and derive four ellipsoidal moment estimates, which we then combine with an elementary anisotropic Gaussian estimate. Throughout this section, all estimates are pointwise in $(t,x)$, which we suppress from the notation. Accordingly, we use the velocity-only weighted norms:
    \begin{align}\label{eq:norm-restric-v}
        \begin{split}
            \|f\|_{L^\infty_q}
		:=
		\operatorname*{ess\,sup}_{v\in\mathbb R^3}
		\langle v\rangle^q |f(v)|, \qquad \|f\|_{L^1_2}:=\int_{\mathbb R^3}
		(1+|v|^2)|f(v)|\,dv.    
        \end{split}
    \end{align}
    
    \begin{lemma}[Strict positive definiteness of the stress tensor]\label{lem:theta-positive-definite}
    Let $f\in L^1_2(\mathbb R^3)$ satisfy $f\ge0$ and
    \begin{align*}
    \rho:=\int_{\mathbb R^3}f(v)\,dv>0.
	\end{align*}
	Then the stress tensor
	\begin{align*}
			\Theta
			=
			\frac1\rho
			\int_{\mathbb R^3}(v-U)\otimes(v-U)f(v)\,dv
		\end{align*}
		is symmetric positive definite.
	\end{lemma}
	
	\begin{proof}
		Symmetry follows immediately from the definition.  For every
		$\xi\in\mathbb R^3$,
		\begin{align*}
			\rho\,\xi^T\Theta\xi
			=
			\int_{\mathbb R^3}\bigl((v-U)\cdot\xi\bigr)^2f(v)\,dv
			\ge0,
		\end{align*}
		so $\Theta$ is positive semi-definite.\\
        \indent Suppose that $\Theta$ is not positive definite. Then, there exists
		$\xi\in\mathbb R^3\setminus\{0\}$ such that
		\begin{align*}
			\int_{\mathbb R^3}\bigl((v-U)\cdot\xi\bigr)^2f(v)\,dv=0.
		\end{align*}
		The integrand is nonnegative, and hence
		\begin{align*}
			\bigl((v-U)\cdot\xi\bigr)^2f(v)=0,
			\quad\text{for almost every }v\in\mathbb R^3.
		\end{align*}
		It follows that $f(v)=0$ for almost every $v$ outside the affine plane
		\begin{align*}
			P_\xi:=\{v\in\mathbb R^3:(v-U)\cdot\xi=0\}.
		\end{align*}
		Since $P_\xi$ has three-dimensional Lebesgue measure zero and $f$ is an
		$L^1$ function,
		\begin{align*}
			\rho
			=
			\int_{\mathbb R^3}f(v)\,dv
			=
			\int_{P_\xi}f(v)\,dv
			=0,
		\end{align*}
		contradicting the assumption $\rho>0$.  Thus $\Theta$ is positive definite.
	\end{proof}
	
    The following proposition provides the ellipsoidal moment estimates needed for the weighted Gaussian bound. In particular, the directional estimate \ref{item:B4} complements the scalar second-moment estimate \ref{item:B2} by coupling the largest principal variance with the determinant of the stress tensor; its proof adapts the Perthame--Pulvirenti splitting argument \cite{perthame1993weighted} to the longest principal direction.
	\begin{proposition}[Four ellipsoidal moment estimates]
		\label{prop:four-moment-estimates}
		Let $f\in L^1_2(\mathbb R^3)$ satisfy $f\ge0$, and $\rho>0$. Then the
		following estimates hold.
		\begin{enumerate}[
			label=\textup{(B\arabic*)},
			ref=\textup{B\arabic*},
			leftmargin=*,
			itemsep=1.1ex
			]
			\item\label{item:B1}
			If $\|f\|_{L^\infty}<\infty$, then
			\begin{align*}
				\frac{\rho}{\sqrt{\det\Theta}}
				\le
				C\|f\|_{L^\infty}.
			\end{align*}
			
			\item\label{item:B2}
			If $q>5$ and $\|f\|_{L^\infty_q}<\infty$, then
			\begin{align*}
				\rho\bigl(|U|^2+\operatorname{tr}\Theta\bigr)^{(q-3)/2}
				\le
				C_q\|f\|_{L^\infty_q}.
			\end{align*}
			
			\item\label{item:B3}
			If $q>1$ and $\|f\|_{L^\infty_q}<\infty$, then
			\begin{align*}
				\frac{\rho|U|^{q+3}}
				{\sqrt{\det\Theta}\bigl(\theta_1+|U|^2\bigr)^{3/2}}
				\le
				C_q\|f\|_{L^\infty_q}.
			\end{align*}
			
			\item\label{item:B4}
			If $q>5$ and $\|f\|_{L^\infty_q}<\infty$, then
			\begin{align*}
				\frac{\rho\theta_1^{q/2}}{\sqrt{\det\Theta}}
				\le
				C_q\|f\|_{L^\infty_q}.
			\end{align*}
		\end{enumerate}
	\end{proposition}
	
	\begin{proof}
        Let $\{e_1,e_2,e_3 \}$ be an orthonormal eigenbasis of $\Theta$, with corresponding eigenvalues $\theta_1\ge\theta_2\ge\theta_3>0$. That is, 
        \begin{align*}
            \begin{split}
                \Theta e_i=\theta_i e_i,\qquad i=1,2,3.
            \end{split}
        \end{align*}
		By the definition of the stress tensor,
		\begin{align*}
			\rho\Theta
			=
			\int_{\mathbb R^3}
			(v-U)\otimes(v-U)f(v)\,dv.
		\end{align*}
		Therefore, for each $i=1,2,3$,
		\begin{align*}
			\rho\theta_i
			&=
			\rho\,e_i^T\Theta e_i
			\\
			&=
			e_i^T
			\left(
			\int_{\mathbb R^3}
			(v-U)\otimes(v-U)f(v)\,dv
			\right)e_i
			\\
			&=
			\int_{\mathbb R^3}
			|(v-U)\cdot e_i|^2f(v)\,dv.
		\end{align*}
        Define the normalized quadratic form
		\begin{align*}
			Q(v)
			:=
			\sum_{i=1}^3\frac{|(v-U)\cdot e_i|^2}{\theta_i}
			=
			(v-U)^T\Theta^{-1}(v-U).
		\end{align*}        
        Then, we have
		\begin{align}
			\int_{\mathbb R^3}Q(v)f(v)\,dv
			=
			\sum_{i=1}^3
			\frac1{\theta_i}
			\int_{\mathbb R^3}
			|(v-U)\cdot e_i|^2f(v)\,dv
			=
			\sum_{i=1}^3
			\frac{\rho\theta_i}{\theta_i}
			=
			3\rho.
			\label{eq:Q-moment}
		\end{align}
		
		\bigskip
		
		\noindent
		\textbf{Proof of \ref{item:B1}.}
		For $L>0$, set
		\begin{align*}
			E_L:=\{v\in\mathbb R^3:Q(v)\le L^2\}.
		\end{align*}
		Under the change of variables
		\begin{align*}
			y_i=\frac{(v-U)\cdot e_i}{\sqrt{\theta_i}},\qquad i=1,2,3,
		\end{align*}
		we have $dv=\sqrt{\det\Theta}\,dy$ and $E_L$ is mapped onto the Euclidean
		ball of radius $L$.  Thus, with
		\begin{align*}
			\omega_3:=|\{y\in\mathbb R^3:|y|\le1\}|,
		\end{align*}
		we have
		\begin{align}
			|E_L|
			=
			\omega_3\sqrt{\det\Theta}\,L^3.
			\label{eq:ellipsoid-volume}
		\end{align}
		We split the mass into $E_L$ and its complement.  Since $Q(v)>L^2$ on
		$E_L^c$,
		\begin{align*}
			\rho
			&=
			\int_{E_L}f(v)\,dv
			+
			\int_{E_L^c}f(v)\,dv
			\\
			&\le
			\|f\|_{L^\infty}|E_L|
			+
			\frac1{L^2}\int_{E_L^c}Q(v)f(v)\,dv
			\\
			&\le
			\omega_3\sqrt{\det\Theta}\,L^3\|f\|_{L^\infty}
			+
			\frac{3\rho}{L^2},
		\end{align*}
		where we used \eqref{eq:Q-moment}.  Because $\rho>0$, one necessarily has
		$\|f\|_{L^\infty}>0$.  We choose $L>0$ so that the two terms on the
		right-hand side are equal:
		\begin{align*}
			\omega_3\sqrt{\det\Theta}\,L^3\|f\|_{L^\infty}
			=
			\frac{3\rho}{L^2}.
		\end{align*}
		Equivalently,
		\begin{align*}
			L^5
			=
			\frac{3\rho}
			{\omega_3\sqrt{\det\Theta}\|f\|_{L^\infty}}.
		\end{align*}
		For this choice,
		\begin{align*}
			\rho
			&\le
			2\omega_3\sqrt{\det\Theta}\,L^3\|f\|_{L^\infty}
			\\
			&=
			2\omega_3\sqrt{\det\Theta}\|f\|_{L^\infty}
			\left(
			\frac{3\rho}
			{\omega_3\sqrt{\det\Theta}\|f\|_{L^\infty}}
			\right)^{3/5}
			\\
			&=
			2\cdot3^{3/5}\rho^{3/5}
			\left(
			\omega_3\sqrt{\det\Theta}\|f\|_{L^\infty}
			\right)^{2/5}.
		\end{align*}
		Dividing by $\rho^{3/5}$ and raising the resulting inequality to the power
		$5/2$ gives
		\begin{align*}
			\rho
			\le
			C\sqrt{\det\Theta}\|f\|_{L^\infty}.
		\end{align*}
		
		\medskip
		\noindent
		\textbf{Proof of \ref{item:B2}}. From the definitions of $U$ and
		$\Theta$, we have
		\begin{align}
			\int_{\mathbb R^3}|v|^2f(v)\,dv
			=
			\rho\bigl(|U|^2+\operatorname{tr}\Theta\bigr).
			\label{eq:raw-second-moment}
		\end{align}
		Fix $R>0$.  Splitting the second moment at the Euclidean radius $R$ gives
		\begin{align*}
			\rho\bigl(|U|^2+\operatorname{tr}\Theta\bigr)
			&=
			\int_{|v|\le R}|v|^2f(v)\,dv
			+
			\int_{|v|>R}|v|^2f(v)\,dv
			\\
			&\le
			\rho R^2
			+
			\|f\|_{L^\infty_q}
			\int_{|v|>R}|v|^{2-q}\,dv.
		\end{align*}
		Since $q>5$,
		\begin{align*}
			\int_{|v|>R}|v|^{2-q}\,dv
			=
			4\pi\int_R^\infty r^{4-q}\,dr
			=
			\frac{4\pi}{q-5}R^{5-q}.
		\end{align*}
		Therefore,
		\begin{align}
			\rho\bigl(|U|^2+\operatorname{tr}\Theta\bigr)
			\le
			\rho R^2
			+
			\frac{4\pi}{q-5}\|f\|_{L^\infty_q}R^{5-q}.
			\label{eq:B2-before-balance}
		\end{align}
		We choose $R>0$ so that the two terms on the right-hand side are equal:
		\begin{align*}
			\rho R^2
			=
			\frac{4\pi}{q-5}\|f\|_{L^\infty_q}R^{5-q}.
		\end{align*}
		Equivalently,
		\begin{align*}
			R^{q-3}
			=
			\frac{4\pi}{q-5}
			\frac{\|f\|_{L^\infty_q}}{\rho}.
		\end{align*}
		Substitution into \eqref{eq:B2-before-balance} yields
		\begin{align*}
			\rho\bigl(|U|^2+\operatorname{tr}\Theta\bigr)\le
			2\rho R^2\le
			C_q
			\rho^{1-2/(q-3)}
			\|f\|_{L^\infty_q}^{2/(q-3)}.
		\end{align*}
		After division by $\rho>0$,
		\begin{align*}
			|U|^2+\operatorname{tr}\Theta
			\le
			C_q
			\left(
			\frac{\|f\|_{L^\infty_q}}{\rho}
			\right)^{2/(q-3)}.
		\end{align*}
		Raising both sides to the power $(q-3)/2$ and multiplying by $\rho$ yields the desired result.
        
		\medskip
		\noindent
		\textbf{Proof of \ref{item:B3}.}
		Let $q>1$. First, we see
		\begin{align*}
			\rho|U|
			=
			\left|\int_{\mathbb R^3}v f(v)\,dv\right|
			\le
			\int_{\mathbb R^3}|v|f(v)\,dv.
		\end{align*}
		We split the last integral over $E_L$ and $E_L^c$.\\
        $\bullet$ Estimate on $E_L:$ H\"older's inequality gives
		\begin{align*}
			\int_{E_L}|v|f(v)\,dv
			\le
			\left(\int_{E_L}f(v)\,dv\right)^{1-1/q}
			\left(\int_{E_L}|v|^qf(v)\,dv\right)^{1/q}
			\le
			\rho^{1-1/q}
			\left(\|f\|_{L^\infty_q}|E_L|\right)^{1/q}.
		\end{align*}
		Here we used
		\begin{align*}
			|v|^qf(v)
			\le
			\langle v\rangle^qf(v)
			\le
			\|f\|_{L^\infty_q}.
		\end{align*}
		By \eqref{eq:ellipsoid-volume},
		\begin{align}
			\int_{E_L}|v|f(v)\,dv
			\le
			\omega_3^{1/q}
			\rho^{1-1/q}
			\bigl(\|f\|_{L^\infty_q}\sqrt{\det\Theta}\bigr)^{1/q}
			L^{3/q}.
			\label{eq:B3-near}
		\end{align}
		$\bullet$ Estimate on $E^c_L:$ we utilize $Q(v)^{1/2}>L$ and the Cauchy--Schwarz inequality to obtain
		\begin{align*}
			\int_{E_L^c}|v|f(v)\,dv\le
			\frac1L
			\int_{\mathbb R^3}Q(v)^{1/2}|v|f(v)\,dv\le
			\frac1L
			\left(\int_{\mathbb R^3}Q(v)f(v)\,dv\right)^{1/2}
			\left(\int_{\mathbb R^3}|v|^2f(v)\,dv\right)^{1/2}.
		\end{align*}
		The first factor is $(3\rho)^{1/2}$ by \eqref{eq:Q-moment}.  By
		\eqref{eq:raw-second-moment} and
		$\theta_1\ge\theta_2\ge\theta_3$,
		\begin{align*}
			\int_{\mathbb R^3}|v|^2f(v)\,dv
			=
			\rho\bigl(|U|^2+\theta_1+\theta_2+\theta_3\bigr)
			\le
			3\rho\bigl(|U|^2+\theta_1\bigr).
		\end{align*}
		Therefore,
		\begin{align}
			\int_{E_L^c}|v|f(v)\,dv
			\le
			\frac{3\rho}{L}\bigl(\theta_1+|U|^2\bigr)^{1/2}.
			\label{eq:B3-far}
		\end{align}
		Combining \eqref{eq:B3-near} and \eqref{eq:B3-far}, we have
		\begin{align}
			\rho|U|
			\le
			\;\omega_3^{1/q}
			\rho^{1-1/q}
			\bigl(\|f\|_{L^\infty_q}\sqrt{\det\Theta}\bigr)^{1/q}
			L^{3/q}
			+
			\frac{3\rho}{L}\bigl(\theta_1+|U|^2\bigr)^{1/2}.
			\label{eq:B3-before-balance}
		\end{align}
		We choose $L>0$ so that the two terms on the right-hand side are equal:
		\begin{align*}
			\omega_3^{1/q}
			\rho^{1-1/q}
			\bigl(\|f\|_{L^\infty_q}\sqrt{\det\Theta}\bigr)^{1/q}
			L^{3/q}=
			\frac{3\rho}{L}\bigl(\theta_1+|U|^2\bigr)^{1/2}.
		\end{align*}
		Equivalently,
		\begin{align*}
			L^{(q+3)/q}
			=
			\frac{3\rho^{1/q}\bigl(\theta_1+|U|^2\bigr)^{1/2}}
			{\omega_3^{1/q}
				\bigl(\|f\|_{L^\infty_q}\sqrt{\det\Theta}\bigr)^{1/q}}.
		\end{align*}
		Using this choice in \eqref{eq:B3-before-balance}, we obtain
		\begin{align*}
			\rho|U|
			\le
			C_q
			\rho^{(q+2)/(q+3)}
			\bigl(\|f\|_{L^\infty_q}\sqrt{\det\Theta}\bigr)^{1/(q+3)}
			\bigl(\theta_1+|U|^2\bigr)^{3/(2(q+3))}.
		\end{align*}
		Raising this inequality to the power $q+3$ gives
		\begin{align*}
			\rho^{q+3}|U|^{q+3}
			\le
			C_q
			\rho^{q+2}
			\|f\|_{L^\infty_q}\sqrt{\det\Theta}
			\bigl(\theta_1+|U|^2\bigr)^{3/2}.
		\end{align*}
		Dividing by
		$\rho^{q+2}\sqrt{\det\Theta}(\theta_1+|U|^2)^{3/2}$ proves the desired result.
		
		\medskip
		
		\noindent
		\textbf{Proof of \ref{item:B4}.}
		Let $q>5$ and assume that $\|f\|_{L^\infty_q}<\infty$.
		We shall estimate the second moment in the longest direction $e_1$. We observe
		\begin{align*}
			\int_{\mathbb R^3}|v\cdot e_1|^2f(v)\,dv
			&=
			\int_{\mathbb R^3}
			\left|(v-U)\cdot e_1+U\cdot e_1\right|^2f(v)\,dv
			\\
			&=
			\int_{\mathbb R^3}
			\bigl((v-U)\cdot e_1\bigr)^2f(v)\,dv
			+
			\rho (U\cdot e_1)^2,
		\end{align*}
		where we used the fact that the cross term vanishes because
		\begin{align*}
			\int_{\mathbb R^3}(v-U)f(v)\,dv=0.
		\end{align*}
		Moreover, by the definition of $\Theta$,
		\begin{align*}
			\int_{\mathbb R^3}
			\bigl((v-U)\cdot e_1\bigr)^2f(v)\,dv
			&=
			\int_{\mathbb R^3}e_1^{T}\left\{(v-U)\otimes(v-U)\right\}e_1f(v)\,dv
			\\
			&=
			e_1^T
			\left(
			\int_{\mathbb R^3}(v-U)\otimes(v-U)f(v)\,dv
			\right)e_1
			\\
			&=
			e_1^{T}\left\{\rho\Theta \right\}e_1
			=
			\rho\theta_1.
		\end{align*}
		Therefore, we get
		\begin{align}
			\int_{\mathbb R^3}|v\cdot e_1|^2f(v)\,dv
			=
			\rho\bigl((U\cdot e_1)^2+\theta_1\bigr)
			\ge
			\rho\theta_1.
			\label{eq:e1-second-moment-direct}
		\end{align}
		We now split the integral according to whether $|v\cdot e_1|$ is small or large.
		The contribution from the set
		\begin{align*}
			\left\{v\in\mathbb R^3: |v\cdot e_1|<\frac{\sqrt{\theta_1}}2\right\}
		\end{align*}
		is bounded by
		\begin{align*}
			\int_{\{|v\cdot e_1|<\sqrt{\theta_1}/2\}}|v\cdot e_1|^2f(v)\,dv
			\le
			\frac{\theta_1}{4}
			\int_{\mathbb R^3}f(v)\,dv
			=
			\frac{\rho\theta_1}{4}.
		\end{align*}
		Combining this with \eqref{eq:e1-second-moment-direct}, we obtain
        \begin{align}\label{eq:e1-tail-direct}
            \begin{split}
                \int_{\{|v\cdot e_1|\ge\sqrt{\theta_1}/2\}}|v\cdot e_1|^2f(v)\,dv&=\int_{\mathbb{R}^3}|v\cdot e_1|^2f(v)\,dv-\int_{\{|v\cdot e_1|<\sqrt{\theta_1}/2\}}|v\cdot e_1|^2f(v)\,dv\cr
                &\ge \frac{3\rho\theta_1}{4}.
            \end{split}
        \end{align}
		On the other hand, we define the transverse normalized
		quadratic form
		\begin{align*}
			Q_\perp(v)
			:=
			\frac{\bigl((v-U)\cdot e_2\bigr)^2}{\theta_2}
			+
			\frac{\bigl((v-U)\cdot e_3\bigr)^2}{\theta_3}.
		\end{align*}
		By the definition of $\Theta$ again, we have
		\begin{align}\label{eq:transverse-energy-direct}
			\int_{\mathbb R^3}Q_\perp(v)f(v)\,dv
			=
			\sum_{i=2}^3
			\frac1{\theta_i}
			\int_{\mathbb R^3}
			\bigl((v-U)\cdot e_i\bigr)^2f(v)\,dv
			=
			\sum_{i=2}^3
			\frac{\rho\theta_i}{\theta_i}
			=
			2\rho.
		\end{align}
		Let $K>0$ be a parameter to be chosen later. We split the tail region in
		\eqref{eq:e1-tail-direct} into
		\begin{align*}
			\left\{
			|v\cdot e_1|\ge\frac{\sqrt{\theta_1}}2,
			\quad
			Q_\perp(v)\le K |v\cdot e_1|^2
			\right\},
		\end{align*}
		and
		\begin{align*}
			\left\{
			|v\cdot e_1|\ge\frac{\sqrt{\theta_1}}2,
			\quad
			Q_\perp(v)>K |v\cdot e_1|^2
			\right\}.
		\end{align*}
		We first estimate the part where $Q_\perp(v)\le K |v\cdot e_1|^2$. Since $|v\cdot e_1|\leq |v|$, we see
		\begin{align*}
			f(v)\le \|f\|_{L^{\infty}_q}|v|^{-q}\le \|f\|_{L^{\infty}_q}|v\cdot e_1|^{-q}.
		\end{align*}
		For each fixed $s\in\mathbb R$, consider the plane
		\begin{align*}
			H_s:=\{v\in\mathbb R^3:v\cdot e_1=s\}.
		\end{align*}
		Inside this plane, the set
		\begin{align*}
			\left\{v\in H_s:Q_\perp(v)\le K s^2\right\}
		\end{align*}
		is an ellipse in the directions $e_2,e_3$ with semi-axes
		\begin{align*}
			\sqrt{\theta_2}\sqrt K\,|s|,
			\qquad
			\sqrt{\theta_3}\sqrt K\,|s|.
		\end{align*}
        Therefore, the two-dimensional area on $H_s$ of this ellipse is 
        \begin{align*}
            \begin{split}
                \operatorname{Area}(\{v\in H_s:Q_{\perp}(v)\leq Ks^2  \})=\pi\sqrt{\theta_2}\sqrt{\theta_3}Ks^2.
            \end{split}
        \end{align*}
        Using Fubini's theorem in the orthonormal basis
		$\{e_1,e_2,e_3\}$, we get
		\begin{align*}
			&\int_{\substack{|v\cdot e_1|\ge\sqrt{\theta_1}/2\\
            Q_\perp(v)\le K |v\cdot e_1|^2}}
			|v\cdot e_1|^2f(v)\,dv\cr
            &=\int_{|s|\geq\sqrt{\theta_1}/2}\left(\int_{\substack{v\in H_s\\ Q_{\perp}(v)\leq Ks^2}}s^2f(v)ds_\perp \right)ds\cr
            &\le
			\|f\|_{L^\infty_q}
			\int_{|s|\ge\sqrt{\theta_1}/2}
			|s|^{2-q} \operatorname{Area}(\{
			v\in H_s:Q_\perp(v)\le K s^2
			\}
			)\,ds\cr
			&=\pi\|f\|_{L^\infty_q}
			\sqrt{\theta_2}\sqrt{\theta_3}K
			\int_{|s|\ge\sqrt{\theta_1}/2}
			|s|^{4-q}\,ds.
		\end{align*}
        Here $ds_{\perp}$ denotes the two-dimensional area measure on $H_s$. Since $q>5$, we get
		\begin{align*}
			\int_{|s|\ge\sqrt{\theta_1}/2}|s|^{4-q}\,ds
			=
			C_q\bigl(\sqrt{\theta_1}\bigr)^{5-q}.
		\end{align*}
		Thus
		\begin{align}
			\int_{\substack{| v\cdot e_1|\ge\sqrt{\theta_1}/2\\ Q_\perp(v)\le K |v\cdot e_1|^2}}
			|v\cdot e_1|^2f(v)\,dv
			\le
			C_q
			\|f\|_{L^\infty_q}
			\sqrt{\theta_2}\sqrt{\theta_3}
			K
			\bigl(\sqrt{\theta_1}\bigr)^{5-q}.
			\label{eq:direct-near-transverse}
		\end{align}
		We now estimate the complementary part. On the set
		\begin{align*}
			Q_\perp(v)>K |v\cdot e_1|^2,
		\end{align*}
		we have
		\begin{align*}
			|v\cdot e_1|^2
			\le
			\frac1K Q_\perp(v).
		\end{align*}
		Therefore, by \eqref{eq:transverse-energy-direct},
		\begin{align}
			\int_{\substack{|v\cdot e_1|\ge\sqrt{\theta_1}/2\\ Q_\perp(v)>K |v\cdot e_1|^2}}
			|v\cdot e_1|^2f(v)\,dv
			\le
			\frac1K
			\int_{\mathbb R^3}Q_\perp(v)f(v)\,dv
			=\frac{2\rho}{K}.
			\label{eq:direct-far-transverse}
		\end{align}
		Combining \eqref{eq:e1-tail-direct}, \eqref{eq:direct-near-transverse}, and
		\eqref{eq:direct-far-transverse}, we find
		\begin{align}
			\rho\theta_1
			\le
			C_q
			\|f\|_{L^\infty_q}
			\sqrt{\theta_2}\sqrt{\theta_3}
			K
			\bigl(\sqrt{\theta_1}\bigr)^{5-q}
			+
			C\frac{\rho}{K}.
			\label{eq:direct-before-balance}
		\end{align}
		We now choose $K>0$ so that the two terms on the right-hand side are
		balanced:
		\begin{align*}
			\|f\|_{L^\infty_q}
			\sqrt{\theta_2}\sqrt{\theta_3}
			K
			\bigl(\sqrt{\theta_1}\bigr)^{5-q}
			=
			\frac{\rho}{K}.
		\end{align*}
		Equivalently,
		\begin{align*}
			K^2
			=
			\frac{\rho}
			{
				\|f\|_{L^\infty_q}
				\sqrt{\theta_2}\sqrt{\theta_3}
				\bigl(\sqrt{\theta_1}\bigr)^{5-q}
			}.
		\end{align*}
		With this choice, \eqref{eq:direct-before-balance} gives
		\begin{align*}
			\rho\theta_1
			\le
			C_q
			\left(
			\rho
			\|f\|_{L^\infty_q}
			\sqrt{\theta_2}\sqrt{\theta_3}
			\bigl(\sqrt{\theta_1}\bigr)^{5-q}
			\right)^{1/2}.
		\end{align*}
		This yields
		\begin{align*}
			\rho
			\bigl(\sqrt{\theta_1}\bigr)^{q-1}
			\le
			C_q
			\|f\|_{L^\infty_q}
			\sqrt{\theta_2}\sqrt{\theta_3}.
		\end{align*}
		Finally, using
		\begin{align*}
			\sqrt{\det\Theta}
			=
			\sqrt{\theta_1}\sqrt{\theta_2}\sqrt{\theta_3},
		\end{align*}
		we obtain
		\begin{align*}
			\frac{\rho\theta_1^{q/2}}{\sqrt{\det\Theta}}=
			\frac{
				\rho(\sqrt{\theta_1})^q
			}{
				\sqrt{\theta_1}\sqrt{\theta_2}\sqrt{\theta_3}
			}
			=
			\frac{
				\rho(\sqrt{\theta_1})^{q-1}
			}{
				\sqrt{\theta_2}\sqrt{\theta_3}
			}
			\le
			C_q\|f\|_{L^\infty_q}.
		\end{align*}
		This proves \eqref{item:B4}.		
	\end{proof}

	\begin{lemma}
		\label{lem:weighted-critical-esbgk-gaussian}
        Let $\rho>0$, $U\in\mathbb{R}^3$, and let $\mathcal T_{-1/2}$ be a symmetric positive definite $3\times 3$ matrix. Define
		\begin{align*}
			\mathcal M_{-1/2}[f](v)
			:=
			\frac{\rho}
			{\sqrt{\det(2\pi\mathcal T_{-1/2})}}
			\exp\left(
			-\frac12
			(v-U)^T\mathcal T_{-1/2}^{-1}(v-U)
			\right).
		\end{align*}
		Then, for every $q>0$,
		\begin{align*}
			\|\mathcal M_{-1/2}[f]\|_{L^\infty_q}
			\le
			C_q
			\frac{\rho}{\sqrt{\det\mathcal T_{-1/2}}}
			\left(
			\lambda_{\max}(\mathcal T_{-1/2})^{q/2}
			+
			\langle U\rangle^q
			\right).
		\end{align*}
	\end{lemma}
	
	\begin{proof}
		By definition of the weighted norm,
		\begin{align*}
			\|\mathcal M_{-1/2}[f]\|_{L^\infty_q}
			&=
			\operatorname*{ess\,sup}_{v\in\mathbb R^3}
			\langle v\rangle^q
			\frac{\rho}
			{\sqrt{\det(2\pi\mathcal T_{-1/2})}}
			\exp\left(
			-\frac12
			(v-U)^T\mathcal T_{-1/2}^{-1}(v-U)
			\right).
		\end{align*}
		Since $\mathcal T_{-1/2}$ is symmetric positive definite, the change of variables
		\begin{align*}
			z
			:=
			\mathcal T_{-1/2}^{-1/2}(v-U),
			\qquad
			v
			=
			U+\mathcal T_{-1/2}^{1/2}z
		\end{align*}
		is one-to-one from $\mathbb R^3$ onto $\mathbb R^3$. Moreover,
		\begin{align*}
			(v-U)^T\mathcal T_{-1/2}^{-1}(v-U)
			&=
			\left(
			\mathcal T_{-1/2}^{1/2}z
			\right)^T
			\mathcal T_{-1/2}^{-1}
			\left(
			\mathcal T_{-1/2}^{1/2}z
			\right)
			\\
			&=
			z^T
			\mathcal T_{-1/2}^{1/2}
			\mathcal T_{-1/2}^{-1}
			\mathcal T_{-1/2}^{1/2}
			z
			\\
			&=
			z^Tz
			=
			|z|^2.
		\end{align*}
		Therefore,
		\begin{align}
			\|\mathcal M_{-1/2}[f]\|_{L^\infty_q}
			&=
			\frac{\rho}
			{\sqrt{\det(2\pi\mathcal T_{-1/2})}}
			\sup_{z\in\mathbb R^3}
			\left\langle
			U+\mathcal T_{-1/2}^{1/2}z
			\right\rangle^q
			e^{-|z|^2/2}.
			\label{eq:gaussian-after-z-change}
		\end{align}
		We now estimate the weight. By the triangle inequality,
		\begin{align*}
			\left\langle
			U+\mathcal T_{-1/2}^{1/2}z
			\right\rangle
			&=
			\left(
			1+
			\left|
			U+\mathcal T_{-1/2}^{1/2}z
			\right|^2
			\right)^{1/2}
			\\
			&\le
			1+|U|+\left|\mathcal T_{-1/2}^{1/2}z\right|
			\\
			&\le
			C
			\left(
			\langle U\rangle
			+
			\left|\mathcal T_{-1/2}^{1/2}z\right|
			\right).
		\end{align*}
		Since
		\begin{align*}
			\left|\mathcal T_{-1/2}^{1/2}z\right|^2=
			z^T\mathcal T_{-1/2}z\le
			\lambda_{\max}(\mathcal T_{-1/2})|z|^2,
		\end{align*}
		we have
		\begin{align*}
			\left|\mathcal T_{-1/2}^{1/2}z\right|
			\le
			\lambda_{\max}(\mathcal T_{-1/2})^{1/2}|z|.
		\end{align*}
		Hence, using $(a+b)^q\le C_q(a^q+b^q)$ for $a,b\ge0$,
		\begin{align}
			\left\langle
			U+\mathcal T_{-1/2}^{1/2}z
			\right\rangle^q
			\le
			C_q
			\left(
			\langle U\rangle^q
			+
			\left|
			\mathcal T_{-1/2}^{1/2}z
			\right|^q
			\right)
			\le
			C_q
			\left(
			\langle U\rangle^q
			+
			\lambda_{\max}(\mathcal T_{-1/2})^{q/2}|z|^q
			\right).
			\label{eq:weight-bound-critical}
		\end{align}
		Substituting \eqref{eq:weight-bound-critical} into
		\eqref{eq:gaussian-after-z-change}, we obtain
		\begin{align*}
			\|\mathcal M_{-1/2}[f]\|_{L^\infty_q}
			&\le
			C_q
			\frac{\rho}
			{\sqrt{\det(2\pi\mathcal T_{-1/2})}}
			\sup_{z\in\mathbb R^3}
			\left(
			\langle U\rangle^q
			+
			\lambda_{\max}(\mathcal T_{-1/2})^{q/2}|z|^q
			\right)
			e^{-|z|^2/2}
			\\
			&\le C_q
			\frac{\rho}
			{\sqrt{\det\mathcal T_{-1/2}}}
			\left(
			\langle U\rangle^q
			+
			\lambda_{\max}(\mathcal T_{-1/2})^{q/2}
			\right),
		\end{align*}
        where we used
        \begin{align*}
            \begin{split}
                \sup_{z\in\mathbb R^3}e^{-|z|^2/2}=1,\qquad \sup_{z\in\mathbb R^3}|z|^q e^{-|z|^2/2}<\infty.   
            \end{split}
        \end{align*}
		This completes the proof.
	\end{proof}

	\bigskip
	\begin{lemma}
		\label{lem:critical-ES-Maxwellian-PP-style}
		Let $q>5$. Assume $f\ge0$, $\|f\|_{L^\infty_q}<\infty$, and $\rho>0$.
		Then, there exists $C_{\mathcal{M}}>0$ depending only on $q$ such that
		\begin{align*}
		    \begin{split}
		      \|\mathcal M_{-1/2}[f]\|_{L^\infty_q}
		\le
		C_{\mathcal{M}}\|f\|_{L^\infty_q}.    
		    \end{split}
		\end{align*}
	\end{lemma}
	
	\begin{proof}
		Recall
        \begin{align*}
            \begin{split}
                \mathcal T_{-1/2}
		=
		\frac12\bigl((\operatorname{tr}\Theta)I-\Theta\bigr).   
            \end{split}
        \end{align*}
		By Lemma~\ref{lem:weighted-critical-esbgk-gaussian}, we have
        \begin{align*}
            \begin{split}
                \|\mathcal M_{-1/2}[f]\|_{L^\infty_q}
		\le
		C_q
		\frac{\rho}{\sqrt{\det \mathcal T_{-1/2}}}
		\left(
		\lambda_{\max}\left(\mathcal T_{-1/2}\right)^{q/2}
		+
		\langle U\rangle^q
		\right).    
            \end{split}
        \end{align*}
		Since $\Theta$ is positive definite, we have
        \begin{align*}
            \begin{split}
                \det  \mathcal{T}_{-1/2}\geq \det\Theta,\qquad \lambda_{\max}(\mathcal{T}_{-1/2})\leq \lambda_{\max}(\Theta)=\theta_1.
            \end{split}
        \end{align*}
		Hence, we get
		\begin{align*}
			\|\mathcal M_{-1/2}[f]\|_{L^\infty_q}
			\le
			C_q
			\frac{\rho}{\sqrt{\det\Theta}}
			\left(
			\theta_1^{q/2}
			+
			\langle U\rangle^q
			\right)\leq C_q\frac{\rho}{\sqrt{\det\Theta}}+ 	C_q\frac{\rho\theta_1^{q/2}}{\sqrt{\det\Theta}}
			+
			C_q\frac{\rho |U|^q}{\sqrt{\det\Theta}}.
		\end{align*}
		The first two terms are bounded using \eqref{item:B1} and \eqref{item:B4} in Proposition~\ref{prop:four-moment-estimates}:
        \begin{align*}
            \begin{split}
             C_q\frac{\rho}{\sqrt{\det\Theta}}+ 	C_q\frac{\rho\theta_1^{q/2}}{\sqrt{\det\Theta}}\leq C_q\|f\|_{L^{\infty}_q}.   
            \end{split}
        \end{align*}
        To control the last term, we split it into two cases. If $|U|\le \sqrt{\theta_1}$, \eqref{item:B4} leads to
        \begin{align*}
            \begin{split}
                \frac{\rho|U|^q}{\sqrt{\det\Theta}}\leq \frac{\rho\theta^{\frac{q}{2}}_1}{\sqrt{\det\Theta}}\leq C_q\|f\|_{L^{\infty}_q}.
            \end{split}
        \end{align*}
        If $|U|>\sqrt{\theta_1}$, we utilize \eqref{item:B3} to get
        \begin{align*}
            \begin{split}
                \frac{\rho |U|^q}{\sqrt{\det\Theta}}
		=
		\frac{\rho |U|^{q+3}}
		{\sqrt{\det\Theta}|U|^3}
		\le
		C
		\frac{\rho |U|^{q+3}}
		{\sqrt{\det\Theta}(\theta_1+|U|^2)^{3/2}}
		\le
		C_q\|f\|_{L^\infty_q}.    
            \end{split}
        \end{align*}
		Combining the two cases gives
        \begin{align*}
            \begin{split}
                C_q\frac{\rho|U|^q}{\sqrt{\det\Theta}}\leq C_q\|f\|_{L^{\infty}_q}.
            \end{split}
        \end{align*}
		This completes the proof.
	\end{proof}
	
	\section{Uniform ellipticity and Lipschitz continuity}
	\label{sec:ellipticity-lipschitz}
	The purpose of this section is to show that the density lower bound and $L^\infty_q$ bound imply uniform ellipticity of $\mathcal T_{-1/2}[f]$. We then use this ellipticity to prove the $L^1_2$-Lipschitz continuity of the critical ES-BGK Maxwellian. Note that we continue to use the velocity-only norms \eqref{eq:norm-restric-v} in Section~\ref{sec:moment-estimates} throughout this section. 
	
	\begin{proposition}
		\label{prop:endpoint-temperature-lower}
		Let $q>5$. Suppose that $f\ge0$, $f\in L^1_2(\mathbb R^3)$, $\|f\|_{L^{\infty}_q}<\infty$, and there is $\rho_*>0$ such that
		\begin{align*}
		    \begin{split}
		      \rho\geq \rho_*.   
		    \end{split}
		\end{align*}
		Then there exist positive constants
		\begin{align*}
		    \begin{split}
		        c_l=c_l(q,\rho_*,\|f\|_{L^{\infty}_q}),\qquad c_u=c_u(q,\rho_*,\|f\|_{L^{\infty}_q}),
		    \end{split}
		\end{align*}
		such that
		\begin{align}\label{eq:macro upper}
		    \begin{split}
		        \rho+|U|+\operatorname{tr}\Theta\leq c_u,
		    \end{split}
		\end{align}
		and
        \begin{align}\label{eq:temperature tensor lower and upper}
            \begin{split}
                c_lI\leq \mathcal{T}_{-1/2}\leq c_uI.
            \end{split}
        \end{align}
		In particular, the scalar temperature also satisfies
		\begin{align*}
		    \begin{split}
		      T=\frac13\operatorname{tr}\Theta\ge
		c_l>0.    
		    \end{split}
		\end{align*}
	\end{proposition}
    \begin{proof}
        First we prove \eqref{eq:macro upper}. Since $q>3$, we have
        \begin{align*}
            \begin{split}
                \rho=\int_{\mathbb{R}^3}f(v)dv\leq \|f\|_{L^{\infty}_q}\int_{\mathbb{R}^3}\left\langle v\right\rangle^{-q}dv\leq C_q\|f\|_{L^{\infty}_q}.
            \end{split}
        \end{align*}
    Next, by \eqref{item:B2} in Proposition~\ref{prop:four-moment-estimates}, we get
    \begin{align*}
        \begin{split}
            |U|^2+\operatorname{tr}\Theta\leq \left(\frac{C_q\|f\|_{L^{\infty}_q}}{\rho_*} \right)^{\frac{2}{q-3}}.
        \end{split}
    \end{align*}
    Therefore, there exists a constant $c_u$ depending on $q$, $\rho_*$, and $\|f\|_{L^{\infty}_q}$, which satisfies \eqref{eq:macro upper}. Now we show \eqref{eq:temperature tensor lower and upper}. Using \eqref{item:B1} in Proposition~\ref{prop:four-moment-estimates}, we see
    \begin{align}\label{eq:stress lower est}
        \begin{split}
            \theta_1\theta_2\theta_3=\det\Theta\geq \left(\frac{\rho}{C\|f\|_{L^{\infty}}}\right)^{2}\geq \left(\frac{\rho_*}{C\|f\|_{L^{\infty}_q}}\right)^2.
        \end{split}
    \end{align}
    Since we have already proved \eqref{eq:macro upper}, the following holds:
    \begin{align*}
        \begin{split}
            \theta_1\leq \operatorname{tr}\Theta\leq c_u,\qquad \theta_2\leq \operatorname{tr}\Theta\leq c_u.
        \end{split}
    \end{align*}
    Using this with \eqref{eq:stress lower est}, $\theta_3$ is controlled by
    \begin{align*}
        \begin{split}
            \theta_3=\frac{\det\Theta}{\theta_1\theta_2}\geq \frac{1}{c^2_u}\left(\frac{\rho_*}{C\|f\|_{L^{\infty}_q}}\right)^2.
        \end{split}
    \end{align*}
    Since $\theta_3$ is the smallest of these eigenvalues of $\Theta$, we deduce
    \begin{align*}
        \begin{split}
            \Theta\geq  \frac{1}{c^2_u}\left(\frac{\rho_*}{C\|f\|_{L^{\infty}_q}}\right)^2I.
        \end{split}
    \end{align*}
    We now pass from $\Theta$ to the critical temperature tensor $\mathcal{T}_{-1/2}=\frac{1}{2}((\operatorname{tr}\Theta)I-\Theta)$. In an eigenbasis of $\Theta$, the eigenvalues of $\mathcal{T}_{-1/2}$ are
    \begin{align*}
        \begin{split}
            \frac{\theta_1+\theta_2}{2},\qquad \frac{\theta_2+\theta_3}{2},\qquad \frac{\theta_1+\theta_3}{2}. 
        \end{split}
    \end{align*}
    Since $\theta_1\geq \theta_2\geq \theta_3$, the smallest of these eigenvalues satisfies
    \begin{align*}
        \begin{split}
            \lambda_{\min}(\mathcal{T}_{-1/2})=\frac{\theta_2+\theta_3}{2}\geq \theta_3\geq \frac{1}{c^2_u}\left(\frac{\rho_*}{C\|f\|_{L^{\infty}_q}}\right)^2.
        \end{split}
    \end{align*}
    If we set
    \begin{align*}
        \begin{split}
            c_l=\frac{1}{c^2_u}\left(\frac{\rho_*}{C\|f\|_{L^{\infty}_q}}\right)^2,
        \end{split}
    \end{align*}
    we conclude
    \begin{align*}
        \begin{split}
            \mathcal{T}_{-1/2}\geq c_lI.
        \end{split}
    \end{align*}
    Moreover, \eqref{eq:macro upper} gives
    \begin{align*}
        \begin{split}
            \lambda_{\max}(\mathcal{T}_{-1/2})=\frac{\theta_1+\theta_2}{2}\leq \operatorname{tr}\Theta\leq c_u.
        \end{split}
    \end{align*}
    Hence, we have
    \begin{align*}
        \begin{split}
            \mathcal{T}_{-1/2}\le c_u I.
        \end{split}
    \end{align*}
    This completes the proof for \eqref{eq:temperature tensor lower and upper}.
	\end{proof}
    
	\begin{lemma}
		\label{lem:macro-lipschitz-L12}
		Let $q>5$.  Suppose that $f,g\ge0$ satisfy
        \begin{align*}
            \begin{split}
                f,g\in L^1_2(\mathbb R^3),
		\qquad
		\rho_f,\rho_g\ge \rho_*>0,
		\qquad
		\|f\|_{L^\infty_q}+\|g\|_{L^\infty_q}<\infty.   
            \end{split}
        \end{align*}
		Then, we have
        \begin{align*}
            \begin{split}
                |\rho_f-\rho_g|+|U_f-U_g|+|\Theta_f-\Theta_g|+|\mathcal T_{-1/2}[f]-\mathcal T_{-1/2}[g]|\le C\|f-g\|_{L^1_2},   
            \end{split}
        \end{align*}
		where $C=C\bigl(q,\rho_*,\max(\|f\|_{L^{\infty}_q},\|g\|_{L^{\infty}_q})\bigr)$.
	\end{lemma}
	
	\begin{proof}
		By definition, we see
        \begin{align*}
            \begin{split}
                |\rho_f-\rho_g|\leq \int_{\mathbb{R}^3}|f-g|dv\leq \|f-g\|_{L^1_2}.
            \end{split}
        \end{align*}
		Next, we have
        \begin{align*}
            \begin{split}
                |U_f-U_g|&=\left|\frac{1}{\rho_f}\int_{\mathbb{R}^3}vf dv-\frac{1}{\rho_g}\int_{\mathbb{R}^3}vgdv \right|\cr
                &\leq \frac{1}{\rho_f}\int_{\mathbb{R}^3}|v||f-g|dv+\frac{\left|\rho_f-\rho_g \right|}{\rho_f\rho_g}\int_{\mathbb{R}^3}|v||g|dv\cr
                &\leq C(q,\rho_*,\|g\|_{L^{\infty}_q})\|f-g\|_{L^1_2}.
            \end{split}
        \end{align*}
		We now estimate the stress tensors. Define the second moment matrices
        \begin{align*}
            \begin{split}
                S_f:=\int_{\mathbb R^3}v\otimes v\,f(v)\,dv,
		\qquad
		S_g:=\int_{\mathbb R^3}v\otimes v\,g(v)\,dv.   
            \end{split}
        \end{align*}
		Then, we have
        \begin{align*}
            \begin{split}
                |S_f-S_g|
		\le
		\int_{\mathbb R^3}|v|^2|f-g|\,dv
		\le \|f-g\|_{L^1_2}.
            \end{split}
        \end{align*}
		Indeed, Proposition~\ref{prop:endpoint-temperature-lower} implies
        \begin{align*}
            \begin{split}
                |S_f|=|\rho_f\Theta_f+\rho_fU_f\otimes U_f|\leq C,\qquad |S_g|=|\rho_g\Theta_g+\rho_gU_g\otimes U_g|\leq C.
            \end{split}
        \end{align*}
        Using these, we control the difference between the stress tensors as follows:
		\begin{align*}
		    \begin{split}
		      |\Theta_f-\Theta_g|
			&\le
			\left|
			\frac{S_f}{\rho_f}
			-
			\frac{S_g}{\rho_g}
			\right|
			+
			|U_f\otimes U_f-U_g\otimes U_g|
			\cr
			&\le
			\frac1{\rho_f}|S_f-S_g|
			+
			\frac{|S_g|}{\rho_f\rho_g}|\rho_f-\rho_g|
			+
			(|U_f|+|U_g|)|U_f-U_g|
			\cr
			&\le
			C\|f-g\|_{L^1_2}. 
		    \end{split}
		\end{align*}
		Finally, the difference estimate for the temperature tensor comes from
        \begin{align*}
            \begin{split}
                |\mathcal{T}_{-1/2}[f]-\mathcal{T}_{-1/2}[g]|=\frac{1}{2}|\operatorname{tr}(\Theta_f-\Theta_g)I-(\Theta_f-\Theta_g)|\leq C|\Theta_f-\Theta_g|\leq C\|f-g\|_{L^1_2}.
            \end{split}
        \end{align*}
		This completes the proof.
	\end{proof}
	
	\begin{proposition}
		\label{prop:critical-maxwellian-L12-lipschitz}
		Let $q>5$.  Suppose that $f,g\ge0$ satisfy
        \begin{align*}
            \begin{split}
                f,g\in L^1_2(\mathbb R^3),
		\qquad
		\rho_f,\rho_g\ge \rho_*>0,
		\qquad
		\|f\|_{L^\infty_q}+\|g\|_{L^\infty_q}<\infty.   
            \end{split}
        \end{align*}
		Then, we have
        \begin{align*}
            \begin{split}
                \|\mathcal M_{-1/2}[f]-\mathcal M_{-1/2}[g]\|_{L^1_2}\le C\|f-g\|_{L^1_2},    
            \end{split}
        \end{align*}
		where $C=C\bigl(q,\rho_*,\max(\|f\|_{L^{\infty}_q},\|g\|_{L^{\infty}_q})\bigr)$.
	\end{proposition}
	
	\begin{proof}
		For simplicity of notation, we set
        \begin{align*}
            \begin{split}
                \mathcal T_f:=\mathcal T_{-1/2}[f],
		\qquad
		\mathcal T_g:=\mathcal T_{-1/2}[g].   
            \end{split}
        \end{align*}
		Now we define the transitional macroscopic fields for $\eta\in[0,1]$
        \begin{align*}
            \begin{split}
                \rho_\eta:=\eta\rho_f+(1-\eta)\rho_g,\quad U_\eta:=\eta U_f+(1-\eta)U_g,\quad \mathcal T_\eta:=\eta\mathcal T_f+(1-\eta)\mathcal T_g.
            \end{split}
        \end{align*}
        By Proposition~\ref{prop:endpoint-temperature-lower}, every convex combination of these macroscopic fields shares
		the same type of bounds. That is,
        \begin{align*}
            \begin{split}
                \rho_\eta\geq \rho_*>0, \qquad |U_\eta|\leq c_u,\qquad c_lI\leq \mathcal{T}_\eta\leq c_uI.
            \end{split}
        \end{align*}
		We define the transitional Gaussian $\mathcal{M}(\eta)$ corresponding to the transitional macroscopic fields
        \begin{align*}
            \begin{split}
                \mathcal M(\eta)(v)
		:=
		\frac{\rho_\eta}{\sqrt{\det(2\pi\mathcal T_\eta)}}
		\exp\left(
		-\frac12
		(v-U_\eta)^T\mathcal T_\eta^{-1}(v-U_\eta)
		\right).   
            \end{split}
        \end{align*}
		Then, Taylor's theorem gives
        \begin{align*}
            \begin{split}
                \mathcal{M}_{-1/2}[f]-\mathcal{M}_{-1/2}[g]=\mathcal{M}(1)-\mathcal{M}(0)=\int^1_0 \frac{d}{d\eta}\mathcal{M}(\eta)d\eta.
            \end{split}
        \end{align*}
        By direct calculation, the chain rule leads to
        \begin{align}\label{eq:eta-derivative-critical}
            \begin{split}
                \frac{d}{d\eta}\mathcal{M}(\eta)&=\frac{\partial \rho_\eta}{\partial \eta}\frac{\partial\mathcal{M}(\eta)}{\partial \rho_\eta}+\frac{\partial U_\eta}{\partial \eta}\cdot \frac{\partial\mathcal{M}(\eta)}{\partial U_\eta}+\frac{\partial \mathcal{T}_\eta}{\partial \eta}:\frac{\partial\mathcal{M}(\eta)}{\partial \mathcal{T}_\eta}\cr
                &=(\rho_f-\rho_g)\frac{\mathcal{M}(\eta)}{\rho_\eta}\cr
                &+\mathcal{M}(\eta)(U_f-U_g)^{T}\mathcal{T}^{-1}_\eta (v-U_\eta)\cr
                &+\frac{1}{2}\mathcal{M}(\eta)\bigl[(v-U_\eta)^{T}\mathcal{T}^{-1}_{\eta}(\mathcal{T}_f-\mathcal{T}_g)\mathcal{T}^{-1}_{\eta}(v-U_\eta)-\operatorname{tr}(\mathcal{T}^{-1}_\eta (\mathcal{T}_f-\mathcal{T}_g)) \bigr]\cr
                &\equiv I_1 + I_2 + I_3.
            \end{split}
        \end{align}
        We now estimate the three terms in
		\eqref{eq:eta-derivative-critical} in $L^1_2$.  We first record a
		uniform Gaussian integral estimate. Set
        \begin{align*}
            \begin{split}
                z:=\mathcal T_\eta^{-1/2}(v-U_\eta),
		\qquad
		v=U_\eta+\mathcal T_\eta^{1/2}z.    
            \end{split}
        \end{align*}
		Then, we have
        \begin{align*}
            \begin{split}
                (v-U_\eta)^T\mathcal T_\eta^{-1}(v-U_\eta)=|z|^2,    
            \end{split}
        \end{align*}
        and
        \begin{align*}
            \begin{split}
                \mathcal{M}(\eta)(v)dv=\rho_\eta(2\pi)^{-\frac{3}{2}}e^{-\frac{|z|^2}{2}}dz.
            \end{split}
        \end{align*}
		Moreover, since $|U_\eta|\le c_u$ and
		$c_lI\leq \mathcal{T}_\eta\le c_uI$, we see
        \begin{align*}
            \begin{split}
                \langle v\rangle^2
		=
		1+|U_\eta+\mathcal T_\eta^{1/2}z|^2
		\le
		C(1+|z|^2),    
            \end{split}
        \end{align*}
        and
        \begin{align*}
            \begin{split}
                |\mathcal{T}_\eta^{-1}(v-U_\eta)|
		=
		|\mathcal{T}_\eta^{-1/2}z|
		\le
		C|z|.   
            \end{split}
        \end{align*}
		For $I_1$, we have
        \begin{align*}
            \begin{split}
               \int_{\mathbb{R}^3}\left\langle v\right\rangle^2|I_1|dv\leq C|\rho_f-\rho_g|\int_{\mathbb{R}^3}(1+|z|^2)e^{-\frac{|z|^2}{2}}dz\leq C|\rho_f-\rho_g|.
            \end{split}
        \end{align*}
        For $I_2$, we compute
        \begin{align*}
            \begin{split}
                \int_{\mathbb{R}^3}\left\langle v\right\rangle^2|I_2|dv\leq C|U_f-U_g|\int_{\mathbb{R}^3}\rho_\eta|z|(1+|z|^2)e^{-\frac{|z|^2}{2}}dz\leq C|U_f-U_g|.     
            \end{split}
        \end{align*}
        Here we used $\rho_\eta\leq c_u$ by Proposition~\ref{prop:endpoint-temperature-lower}.\\
        To obtain the estimate for $I_3$, we additionally observe the following via $\mathcal{T}_\eta\geq c_l I$:
        \begin{align*}
            \begin{split}
                |\operatorname{tr}(\mathcal{T}^{-1}_\eta(\mathcal{T}_f-\mathcal{T}_g))|\leq \left(\sum_{i,j}|(\mathcal{T}^{-1}_\eta)_{ij}|^2 \right)^{\frac{1}{2}}\left(\sum_{i,j}|(\mathcal{T}_f-\mathcal{T}_g)_{ij}|^2\right)^{\frac{1}{2}}\leq C|\mathcal{T}_f-\mathcal{T}_g|,
            \end{split}
        \end{align*}
       and
        \begin{align*}
            \begin{split}
                |(v-U_\eta)^{T}\mathcal{T}^{-1}_{\eta}(\mathcal{T}_f-\mathcal{T}_g)\mathcal{T}^{-1}_{\eta}(v-U_\eta)|=|z^T\mathcal{T}^{-1/2}_\eta(\mathcal{T}_f-\mathcal{T}_g)\mathcal{T}^{-1/2}_\eta z|\leq C|z|^2|\mathcal{T}_f-\mathcal{T}_g|.    
            \end{split}
        \end{align*}
        Using these, we have
        \begin{align*}
            \begin{split}
                \int_{\mathbb{R}^3}\left\langle v\right\rangle^2|I_3|dv&\leq C|\mathcal{T}_f-\mathcal{T}_g|\int_{\mathbb{R}^3}(1+|z|^2)^2e^{-\frac{|z|^2}{2}}dz\leq C|\mathcal{T}_f-\mathcal{T}_g|.  
            \end{split}
        \end{align*}
        Combining the estimates for $I_1$, $I_2$, and $I_3$, we conclude
		\begin{align*}
		    \begin{split}
		      \left\|\frac{d}{d\eta}\mathcal M(\eta)\right\|_{L^1_2}\leq C\left(|\rho_f-\rho_g|+|U_f-U_g|+|\mathcal T_f-\mathcal{T}_g|\right).   
		    \end{split}
		\end{align*}
		Returning to Taylor's formula, Lemma~\ref{lem:macro-lipschitz-L12} gives the desired result as follows:
        \begin{align*}
            \begin{split}
                \|\mathcal M_{-1/2}[f]-\mathcal M_{-1/2}[g]\|_{L^1_2}\leq \int^1_0  \left\|
			\frac{d}{d\eta}\mathcal M(\eta)
			\right\|_{L^1_2}d\eta\leq C\|f-g\|_{L^1_2}. 
            \end{split}
        \end{align*}
	\end{proof}
	\section{Construction of the mild solution}
	\label{sec:endpoint-mild-solutions}
	
	In this section, we indicate how the estimates obtained above give the standard
	mild-solution construction. The existence and uniqueness of a mild solution can be established in both the periodic and whole space settings under the corresponding assumptions on the initial data. We present the proof in $\mathbb{R}^3_x\times\mathbb{R}^3_v$, where the treatment of spatial decay requires additional care. The periodic case follows by a similar argument, with the spatially weighted estimates omitted.
    
    First, we design the solution space $\Omega$ with its metric $d$ as follows:  
    \begin{align*}
        \begin{split}
            \Omega=\{f\in C\bigl([0,T_*];L^1_2(\mathbb{R}^6) \bigr)\,:\,f\text{ satisfies }(A_1)-(A_3) \},\quad d(f,g)=\sup_{0\leq t\leq T_*}\|f-g\|_{L^1_2},
        \end{split}
    \end{align*}
    where conditions $(A_1)-(A_3)$ are given by
    \begin{itemize}
        \item $(A_1):$ $f$ is nonnegative.
        \item $(A_2):$ There exists $C_i>0\;(i=1,2,3)$ such that
        \begin{align*}
            \begin{split}
                \|f(t)\|_{L^{\infty}_{q+r}}&\leq e^{C_1t}\|f_0\|_{L^{\infty}_{q+r}},\cr
                \|f(t)\|_{L^{\infty}_{r,q}}&\leq C_2 e^{C_3 t}\{\|f_0\|_{L^{\infty}_{q+r}}+\|f_0\|_{L^{\infty}_{r,q}} \},
            \end{split}
        \end{align*}
        where $C_{i}\;(i=1,2,3)$ are defined by
        \begin{align*}
            \begin{split}
                C_1=C_{\mathcal{M}}-1,\qquad C_2=C_r(1+T_*)^r,\qquad C_3=2C_{\mathcal{M}}C_2-1.
            \end{split}
        \end{align*}
        \item $(A_3):$ The macroscopic fields satisfy the following:
         \begin{align*}
            \begin{split}
                &\rho(t,x)\geq e^{-T_*}\rho_0\langle x\rangle^{-r},\cr
                &\|\rho(t)\|_{L^{\infty}}+\|U(t)\|_{L^{\infty}}+\|\operatorname{tr}\Theta\|_{L^{\infty}}\leq c_u,\cr
                &c_lI\leq \mathcal{T}_{-1/2}\leq c_uI,
            \end{split}
        \end{align*}
        for some positive constants $c_l$ and $c_u$ depending on $q$, $r$, $\rho_0$, $T_*$, $\|f_0\|_{L^{\infty}_{q+r}}$, and  $\|f_0\|_{L^{\infty}_{r,q}}$. 
    \end{itemize}

    We construct the solution to \eqref{eq:endpoint-esbgk-eq} in $(\Omega,d)$ as the limit function of the sequence $\{f^n\}$ defined by the Picard-type iteration for $n\geq 1$ as follows:
    \begin{align}\label{eq:ESBGK picard eq}
        \begin{split}
            &\partial_t f^{n+1}+v\cdot \nabla_x f^{n+1}=\mathcal{M}_{-1/2}[f^n]-f^{n+1},\cr
            &\partial_t f^1+v\cdot \nabla_x f^1+f^1=0,\qquad f^1(0)=f_0,\cr
            &f^{n+1}(0)=f_0.
        \end{split}
    \end{align}
    To prove Theorem~\ref{thm:endpoint-mild-solution}, we plan to show that $f^n\in\Omega$ for all $n$ and $\{f^n \}$ is a Cauchy sequence in $(\Omega,d)$. Throughout this section, we denote the macroscopic fields associated with $f^n$ by $(\rho^n,U^n,T^n,\Theta^n,\mathcal{T}^n)$ to avoid confusion.
    \begin{proposition}\label{prop:mild sol invariant}
        Assume that $f_0$ satisfies the hypothesis of Theorem~\ref{thm:endpoint-mild-solution}. Let $\{f^n\}$ be a solution to \eqref{eq:ESBGK picard eq}. Then, $f^n$ belongs to $\Omega$ for all $n$.   
    \end{proposition}
    \begin{proof}
        We apply the inductive argument. To do this, we assume $f^n\in\Omega$.\\
        $\bullet$ \textbf{Condition $(A_1)$:} We integrate \eqref{eq:ESBGK picard eq} along the characteristic to get
        \begin{align}\label{eq:Mild form picard}
            \begin{split}
                f^{n+1}(t,x,v)=e^{-t}f_0(x-vt,v)+\int^{t}_0 e^{-(t-s)}\mathcal{M}_{-1/2}[f^n](s,x-(t-s)v,v)ds.
            \end{split}
        \end{align}
        Then, the non-negativity of $f^{n+1}$ directly follows from $f_0\geq 0$ and $f^n\geq 0$.\\
        $\bullet$ \textbf{Condition $(A_2)$:} Multiplying $\langle v\rangle^{q+r}$ by \eqref{eq:Mild form picard}, we see
        \begin{align}\label{eq:condition A2 invariant 1}
            \begin{split}
             |\langle v\rangle^{q+r}f^{n+1}(t,x,v)|&\leq e^{-t}\langle v\rangle^{q+r}|f_0(x-vt,v)|\cr
             &+\int^t_0 e^{-(t-s)}\langle v\rangle^{q+r}|\mathcal{M}_{-1/2}[f^n](s,x-v(t-s),v)|ds.    
            \end{split}
        \end{align}
        Then, Lemma~\ref{lem:critical-ES-Maxwellian-PP-style} leads to
        \begin{align*}
            \begin{split}
                \int^t_0 e^{-(t-s)}\langle v\rangle^{q+r}|\mathcal{M}_{-1/2}[f^n](s,x-v(t-s),v)|ds&\leq \int^t_0 e^{-(t-s)}\|\mathcal{M}_{-1/2}[f^n](s)\|_{L^{\infty}_{q+r}}ds\cr
                &\leq C_{\mathcal{M}}e^{-t}\int^t_0 e^{s}\|f^n(s)\|_{L^{\infty}_{q+r}}ds\cr
                &\leq  C_{\mathcal{M}} e^{-t}\|f_0\|_{L^{\infty}_{q+r}}\int^t_0 e^{(C_1+1)s}ds\cr
                &=\frac{ C_{\mathcal{M}}}{C_1+1}(e^{C_1t}-e^{-t})\|f_0\|_{L^{\infty}_{q+r}}.
            \end{split}
        \end{align*}
        Inserting this estimate into \eqref{eq:condition A2 invariant 1}, we get
        \begin{align*}
            \begin{split}
                |\langle v\rangle^{q+r}f^{n+1}(t,x,v)|\leq e^{-t}\|f_0\|_{L^{\infty}_{q+r}}+\frac{ C_{\mathcal{M}}}{C_1+1}(e^{C_1t}-e^{-t})\|f_0\|_{L^{\infty}_{q+r}}\leq e^{C_1t}\|f_0\|_{L^{\infty}_{q+r}},
            \end{split}
        \end{align*}
        where we used $\langle v\rangle^{q+r}|f_0(x-vt,v)|\leq \|f_0\|_{L^{\infty}_{q+r}}$ and choice of $C_1=C_{\mathcal{M}}-1$. Then, we take the supremum in $(x,v)\in\mathbb{R}^6$ for both sides to get
        \begin{align*}
            \begin{split}
                \|f^{n+1}\|_{L^{\infty}_{q+r}}\leq e^{C_1t}\|f_0\|_{L^{\infty}_{q+r}}.   
            \end{split}
        \end{align*}
        Now we multiply $\langle x\rangle^r\langle v\rangle^q$ by \eqref{eq:Mild form picard}. Then, we have
        \begin{align}\label{eq:condition A2 invariant 2}
            \begin{split}
                |\langle x\rangle^r\langle v\rangle^{q}f^{n+1}(t,x,v)|&\leq e^{-t}\langle x\rangle^r\langle v\rangle^{q}|f_0(x-vt,v)|\cr
             &+\int^t_0 e^{-(t-s)}\langle x\rangle^r\langle v\rangle^{q}|\mathcal{M}_{-1/2}[f^n](s,x-v(t-s),v)|ds.    
            \end{split}
        \end{align}
        For $0\leq s\leq t\leq T_*$, the spatial weight can be computed by
        \begin{align*}
            \begin{split}
             \langle x\rangle^r&\leq C_r(\langle x-v(t-s)\rangle^r+(t-s)^r\langle v\rangle^r)\cr
             &\leq C_r(1+T_*)^r(\langle x-v(t-s)\rangle^r+\langle v\rangle^r)\cr
             &\equiv C_2 (\langle x-v(t-s)\rangle^r+\langle v\rangle^r). 
            \end{split}
        \end{align*}
        Using this, \eqref{eq:condition A2 invariant 2} can be written as
        \begin{align}\label{eq:condition A2 invariant 3}
            \begin{split}
                &|\langle x\rangle^r\langle v\rangle^{q}f^{n+1}(t,x,v)|\cr
                &\leq C_2e^{-t}(\langle x-vt\rangle^r\langle v\rangle^{q}+\langle v\rangle^{q+r})|f_0(x-vt,v)|\cr
             &+C_{2}\int^t_0 e^{-(t-s)}(\langle x-v(t-s)\rangle^r\langle v\rangle^{q}+\langle v\rangle^{q+r})|\mathcal{M}_{-1/2}[f^n](s,x-v(t-s),v)|ds.      
            \end{split}
        \end{align}
        The first term of RHS in \eqref{eq:condition A2 invariant 3} is controlled by
        \begin{align}\label{eq:1st term inv}
            \begin{split}
                (\text{The first term})\leq C_{2}e^{-t}(\|f_0\|_{L^{\infty}_{r,q}}+\|f_0\|_{L^{\infty}_{q+r}}).
            \end{split}
        \end{align}
        On the other hand, Lemma~\ref{lem:critical-ES-Maxwellian-PP-style} gives
        \begin{align*}
            \begin{split}
                \langle x\rangle^r \langle v\rangle^q|\mathcal{M}_{-1/2}[f](t,x,v)|\leq \langle x\rangle^r\sup_{v\in\mathbb{R}^3}\langle v\rangle^q|\mathcal{M}_{-1/2}[f](t,x,v)|\leq C_{\mathcal{M}}\langle x\rangle^r\sup_{v\in\mathbb{R}^3}|\langle v\rangle^qf(t,x,v)|.
            \end{split}
        \end{align*}
        Taking $\sup_x$ on both sides, we get
        \begin{align*}
            \begin{split}
                \|\mathcal{M}_{-1/2}[f](t)\|_{L^{\infty}_{r,q}}\leq C_{\mathcal{M}}\|f(t)\|_{L^{\infty}_{r,q}}.   
            \end{split}
        \end{align*}
        Therefore, the second term of the RHS in \eqref{eq:condition A2 invariant 3} is derived by
        \begin{align}\label{eq:2nd term inv}
            \begin{split}
                &(\text{The 2nd term})\cr
                &\leq C_{\mathcal{M}}C_2\int^{t}_0 e^{-(t-s)}(\|f^n(s)\|_{L^{\infty}_{r,q}}+\|f^n(s)\|_{L^{\infty}_{q+r}})ds\cr
                &\leq 2C_{\mathcal{M}}C^2_{2}e^{-t}(\|f_0\|_{L^{\infty}_{r,q}}+\|f_0\|_{L^{\infty}_{q+r}})\int^t_0 e^{(C_3+1)s}ds\cr
                &=\frac{2C_{\mathcal{M}}C^2_{2}}{C_3+1}(e^{C_3t}-e^{-t})(\|f_0\|_{L^{\infty}_{r,q}}+\|f_0\|_{L^{\infty}_{q+r}}),
            \end{split}
        \end{align}
        where we used $C_1\leq C_3$ and $C_2\geq 1$. Plugging \eqref{eq:1st term inv} and \eqref{eq:2nd term inv} into \eqref{eq:condition A2 invariant 3}, we obtain
        \begin{align*}
            \begin{split}
             \|f^{n+1}(t)\|_{L^{\infty}_{r,q}}\leq C_2e^{C_3t}(\|f_0\|_{L^{\infty}_{r,q}}+\|f_0\|_{L^{\infty}_{q+r}}).   
            \end{split}
        \end{align*}
    $\bullet$ \textbf{Condition $(A_3)$:} Recalling \eqref{eq:Mild form picard}, we see
    \begin{align*}
        \begin{split}
            \rho^{n+1}(t,x)= \int_{\mathbb{R}^3} f^{n+1}(t,x,v)dv\geq e^{-t}\int_{\mathbb{R}^3}f_0(x-vt,v)dv\geq e^{-T_*}\rho_0\langle x\rangle^{-r}.  
        \end{split}
    \end{align*}
    The uniform upper bound of density follows directly from
    \begin{align*}
        \begin{split}
            \rho^{n+1}(t,x)\leq \|f^{n+1}(t)\|_{L^{\infty}_q}\int_{\mathbb{R}^3}\langle v\rangle^{-q}dv\leq C_q\|f^{n+1}(t)\|_{L^{\infty}_{q+r}}\leq C_qe^{C_1T_*}\|f_0\|_{L^{\infty}_{q+r}}.
        \end{split}
    \end{align*}
    According to Proposition~\ref{prop:endpoint-temperature-lower}, we yield
    \begin{align*}
        \begin{split}
            |U^{n+1}|^2+|\operatorname{tr}\Theta^{n+1}|&\leq \left( \frac{C_q\sup_v|\langle v\rangle^qf^{n+1}|}{\rho_0 e^{-T_*}\langle x\rangle^{-r}}\right)^{\frac{2}{q-3}}\cr
            &\leq \left( \frac{C_q\sup_{x,v}|\langle x\rangle^{r}\langle v\rangle^qf^{n+1}|}{\rho_0 e^{-T_*}}\right)^{\frac{2}{q-3}}\cr
            &\leq \left( \frac{C_q\|f^{n+1}\|_{L^{\infty}_{r,q}}}{\rho_0 e^{-T_*}}\right)^{\frac{2}{q-3}}\cr
            &\leq \left( \frac{C_{q,r,T_*}(\|f_0\|_{L^{\infty}_{r,q}}+\|f_0\|_{L^{\infty}_{q+r}})}{\rho_0 e^{-T_*}}\right)^{\frac{2}{q-3}}. 
        \end{split}
    \end{align*}
    Therefore, there is a uniform constant $c_u>0$ which depends on $q$, $r$, $T_*$, $\rho_0$, and $f_0$. Now we apply Proposition~\ref{prop:endpoint-temperature-lower} again to obtain
    \begin{align*}
        \begin{split}
            \lambda_{\max}(\mathcal{T}^{n+1}_{-1/2})\leq \operatorname{tr}\Theta^{n+1}\leq c_u,
        \end{split}
    \end{align*}
    and
    \begin{align*}
        \begin{split}
            \lambda_{\min}(\mathcal{T}^{n+1}_{-1/2})&\geq \frac{e^{-2T_*}\rho^2_0\langle x\rangle^{-2r}}{Cc^2_u\sup_v|\langle v\rangle^qf^{n+1}(t)|^2}\cr
            &\geq \frac{e^{-2T_*}\rho^2_0}{Cc^2_u\sup_{x,v}|\langle x\rangle^r\langle v\rangle^qf^{n+1}(t)|^2}\cr
            &\geq \frac{e^{-2T_*}\rho^2_0}{Cc^2_u\|f^{n+1}\|^2_{L^{\infty}_{r,q}}}\cr
            &\geq \frac{e^{-2T_*}\rho^2_0}{C_{r,q,T_*,\rho_0,f_0}(\|f_0\|^2_{L^{\infty}_{r,q}}+\|f_0\|^2_{L^{\infty}_{q+r}})}.
        \end{split}
    \end{align*}
    These imply that there is a uniform constant $c_l>0$ such that
    \begin{align*}
        \begin{split}
            c_l I \leq \mathcal{T}^{n+1}_{-1/2}\leq c_u I.
        \end{split}
    \end{align*}

    It remains to show that $f^1\in\Omega$. We can write $f^1$ by an explicit formula 
    \begin{align*}
        \begin{split}
            f^1(t,x,v)=e^{-t}f_0(x-vt,v).
        \end{split}
    \end{align*}
    Thus, the non-negativity trivially follows. The weighted $L^{\infty}$ boundedness can be proved by
    \begin{align*}
        \begin{split}
            \langle v\rangle^{q+r}|f^1(t,x,v)|&\le e^{-t} \langle v\rangle^{q+r}|f_0(x-vt,v)|\leq \|f_0\|_{L^{\infty}_{q+r}},
        \end{split}
    \end{align*}
    and
    \begin{align*}
        \begin{split}
             \langle x\rangle^r\langle v\rangle^q|f^1(t,x,v)|&\le C_{r,T_*} (\langle x-vt\rangle^r\langle v\rangle^q+ \langle v\rangle^{q+r})|f_0(x-vt,v)|\leq C_2(\|f_0\|_{L^{\infty}_{r,q}}+\|f_0\|_{L^{\infty}_{q+r}}).   
        \end{split}
    \end{align*}
    Lastly, we have
    \begin{align*}
        \begin{split}
            \rho^1(t,x)=\int_{\mathbb{R}^3} f^1(t,x,v)dv=e^{-t}\int_{\mathbb{R}^3}f_0(x-vt,v)dv\geq e^{-T_*}\rho_0\langle x\rangle^{-r}.
        \end{split}
    \end{align*}
    Therefore, we derive the boundedness of the macroscopic fields associated with $f^1$ via an analogous argument in the inductive step.
    \end{proof}
    
    \begin{proposition}\label{prop:mild sol cauchy}
        Assume that $f_0$ satisfies the hypothesis of Theorem~\ref{thm:endpoint-mild-solution}. Let $\{f^n\}$ be a solution to \eqref{eq:ESBGK picard eq}. Then, $\{f^n\}$ is a Cauchy sequence.   
    \end{proposition}
    \begin{proof}
        Let $F^{n+1}=f^{n+1}-f^n$. Then, the difference between $(n+1)$-th step and $n$-th step for \eqref{eq:Mild form picard} is written as
        \begin{align*}
            \begin{split}
                &\partial_t F^{n+1}+v\cdot \nabla_x F^{n+1}+F^{n+1}=\mathcal{M}_{-1/2}[f^{n}]-\mathcal{M}_{-1/2}[f^{n-1}],\cr
                &F^{n+1}(0)=f^{n+1}(0)-f^n(0)=0.
            \end{split}
        \end{align*}
        Integrating along the characteristic, we have
        \begin{align*}
            \begin{split}
                F^{n+1}(t,x,v)=\int^t_0 e^{-(t-s)}(\mathcal{M}_{-1/2}[f^n]-\mathcal{M}_{-1/2}[f^{n-1}])(s,x-(t-s)v,v)ds.
            \end{split}
        \end{align*}
        We first establish a Lipschitz estimate with a constant independent of the spatial variable and the iteration index. For fixed $s\in[0,T_*]$ and almost every $x\in\mathbb R^3$, Proposition~\ref{prop:mild sol invariant} gives, for $k=n,n-1$,
        \begin{align*}
            \langle x\rangle^r\rho^k(s,x)
            &\ge e^{-T_*}\rho_0,\\
            \|\langle x\rangle^r f^k(s,x,\cdot)\|_{L^\infty_q}
            &\le C_2e^{C_3T_*}
            \left(
                \|f_0\|_{L^\infty_{r,q}}
                +\|f_0\|_{L^\infty_{q+r}}
            \right).
        \end{align*}
We therefore apply
Proposition~\ref{prop:critical-maxwellian-L12-lipschitz}
to $\langle x\rangle^r f^n(s,x,\cdot)$ and
$\langle x\rangle^r f^{n-1}(s,x,\cdot)$,
with a constant independent of $n$, $s$, and $x$.
Using $\mathcal M_{-1/2}[ah]=a\mathcal M_{-1/2}[h]$ for $a>0$ and canceling the common factor $\langle x\rangle^r$, we obtain
\begin{align}\label{eq:picard-uniform-lipschitz}
    \begin{split}
        &\int_{\mathbb R^3}\langle v\rangle^2
        \left|
            \mathcal M_{-1/2}[f^n](s,x,v)
            -\mathcal M_{-1/2}[f^{n-1}](s,x,v)
        \right|\,dv\\
        &\qquad\le
        C\int_{\mathbb R^3}\langle v\rangle^2
        |f^n(s,x,v)-f^{n-1}(s,x,v)|\,dv.
    \end{split}
\end{align}
The same argument applies to any two functions in $\Omega$, with the same constant. Taking absolute values, multiplying by $\langle v\rangle^2$, and integrating in $(x,v)$, we use spatial translation invariance and \eqref{eq:picard-uniform-lipschitz} to obtain
        \begin{align*}
            \begin{split}
                &\int_{\mathbb{R}^6}\langle v\rangle^2|F^{n+1}(t,x,v)|dxdv\cr
                &\leq \int^t_0 e^{-(t-s)}\int_{\mathbb{R}^3}\int_{\mathbb{R}^3}\langle v\rangle^2|(\mathcal{M}_{-1/2}[f^{n}]-\mathcal{M}_{-1/2}[f^{n-1}])(s,x-(t-s)v,v)|dvdxds\cr
                &\leq \int^t_0 \int_{\mathbb{R}^3}\int_{\mathbb{R}^3}\langle v\rangle^2|(\mathcal{M}_{-1/2}[f^{n}]-\mathcal{M}_{-1/2}[f^{n-1}])(s,x,v)|dvdxds\cr
                &\leq C\int^t_0\int_{\mathbb{R}^6}\langle v\rangle^2 |(f^n-f^{n-1})(s,x,v)| dxdvds\cr
                &\leq C\int^t_0\|F^{n}(s)\|_{L^1_2}ds,
            \end{split}
        \end{align*}
        where $C$ is some constant depending on $q$, $r$, $T_*$, $\rho_0$, and $f_0$. We iterate the above estimate to deduce
        \begin{align}\label{eq:iterate}
            \begin{split}
                \sup_{0\leq t\leq T_*}\|F^{n+1}(t)\|_{L^1_2}&\leq \frac{(CT_*)^{n-1}}{(n-1)!}\sup_{0\leq s
                \leq T_*}\|F^2(s)\|_{L^1_2}\cr
                &\leq \frac{(CT_*)^{n-1}}{(n-1)!}\sup_{0\leq s\leq T_*}(\|f^2(s)\|_{L^1_2}+\|f^1(s)\|_{L^1_2}).
            \end{split}
        \end{align}
        Since $f^1$ and $f^2$ belong to $\Omega$, for $i=1,2$, we have
        \begin{align*}
            \begin{split}
                \int_{\mathbb{R}^6}\langle v\rangle^2|f^i(t,x,v)|dxdv\leq \|f^i\|_{L^{\infty}_{r,q}}\int_{\mathbb{R}^6}\langle x \rangle^{-r}\langle v\rangle^{2-q}dxdv\leq C_{r,q}C_2e^{C_3T_*}(\|f_0\|_{L^{\infty}_{r,q}}+\|f_0\|_{L^{\infty}_{q+r}}).
            \end{split}
        \end{align*}
        Applying this to \eqref{eq:iterate}, we get
        \begin{align*}
            \begin{split}
                \sup_{0\leq t\leq T_*}\|F^{n+1}(t)\|_{L^1_2}\leq C_{q,r,T_*,\rho_0,f_0}\frac{(CT_*)^{n-1}}{(n-1)!}.
            \end{split}
        \end{align*}
        Therefore, as $m,n\rightarrow \infty$, we have
        \begin{align*}
            \begin{split}
                \sup_{0\leq t\leq T_*}\|(f^m-f^n)(t)\|_{L^1_2}\leq C_{q,r,T_*,\rho_0,f_0}\left(e^{CT_*}-\sum^{n-2}_{k=0}\frac{(CT_*)^k}{k!} \right)\rightarrow 0.
            \end{split}
        \end{align*}
        This completes the proof.
    \end{proof}

    \noindent\textbf{Proof of Theorem~\ref{thm:endpoint-mild-solution}}. By Proposition~\ref{prop:mild sol cauchy} and the completeness of $C([0,T_*];L^1_2(\mathbb R^6))$, there exists
$f\in C([0,T_*];L^1_2(\mathbb R^6))$ such that
\begin{align*}
    \sup_{0\le t\le T_*}
    \|f^n(t)-f(t)\|_{L^1_2}\longrightarrow0.
\end{align*}
In particular, $f(0)=f_0$.

Fix $t\in[0,T_*]$. Passing to a subsequence, we have
$f^n(t,x,v)\to f(t,x,v)$ for almost every $(x,v)$.
Hence nonnegativity and the weighted $L^\infty$ bounds pass to the limit.
Moreover, convergence in $L^1_2$ implies strong $L^1_x$ convergence
of the density, momentum, and second-moment matrix.
After passing to a further subsequence, these moments converge
for almost every $x$. The density lower bound therefore passes
to the limit. Since $\rho_f(t,x)>0$, the definitions of the
macroscopic fields imply
\begin{align*}
    U^n(t,x)\to U_f(t,x),\qquad
    \Theta^n(t,x)\to\Theta_f(t,x),\qquad
    \mathcal T^n(t,x)\to\mathcal T_{-1/2}[f](t,x)
\end{align*}
for almost every $x$. Thus the remaining bounds in $(A_3)$
also pass to the limit.
Since $t$ was arbitrary, we conclude that $f\in\Omega$ and that
\eqref{eq:endpoint-solution-Linfty-bound},
\eqref{eq:endpoint-solution-density-lower}, and
\eqref{eq:endpoint-solution-ellipticity} hold.

We now pass to the limit in the nonlinear term.
The spatially uniform Lipschitz estimate
\eqref{eq:picard-uniform-lipschitz} applies to any two functions
in $\Omega$. Integrating this estimate in $x$, we obtain
\begin{align*}
    \sup_{0\le t\le T_*}
    \|\mathcal M_{-1/2}[f^n](t)
      -\mathcal M_{-1/2}[f](t)\|_{L^1_2}\le
    C\sup_{0\le t\le T_*}
    \|f^n(t)-f(t)\|_{L^1_2}
    \longrightarrow0.
\end{align*}
Since free transport preserves the global $L^1_2$ norm,
we may pass to the limit in \eqref{eq:Mild form picard}
and obtain
\begin{align}\label{eq:endpoint-mild-form-final}
    \begin{split}
        f(t,x,v)
        =
        e^{-t}f_0(x-vt,v)+
        \int_0^t e^{-(t-s)}
        \mathcal M_{-1/2}[f](s,x-(t-s)v,v)\,ds.
    \end{split}
\end{align}
This identity holds in $L^1_2(\mathbb R^6)$ for every
$t\in[0,T_*]$, and hence for almost every $(x,v)$.
This proves existence.

We next prove uniqueness in $\Omega$.
Let $f,g\in\Omega$ be two mild solutions with the same initial
datum $f_0$. Subtracting their mild formulations, using spatial
translation invariance, and applying the same global Lipschitz
estimate, we obtain
\begin{align*}
    \|f(t)-g(t)\|_{L^1_2}
    \le
    C\int_0^t e^{-(t-s)}
    \|f(s)-g(s)\|_{L^1_2}\,ds
    \le
    C\int_0^t
    \|f(s)-g(s)\|_{L^1_2}\,ds.
\end{align*}
Gronwall's lemma yields $f=g$ on $[0,T_*]$.

It remains to verify the conservation laws and the entropy dissipation. We establish the conservation laws directly from the mild formulation. Writing $\mathcal T_f:=\mathcal T_{-1/2}[f]$, the Gaussian moment identities give
\begin{align*}
    \int_{\mathbb R^3}\mathcal M_{-1/2}[f](v)\,dv
    &=\rho_f,\\
    \int_{\mathbb R^3}v\mathcal M_{-1/2}[f](v)\,dv
    &=\rho_fU_f,\\
    \int_{\mathbb R^3}|v|^2\mathcal M_{-1/2}[f](v)\,dv
    &=\rho_f\bigl(|U_f|^2+\operatorname{tr}\mathcal T_f\bigr).
\end{align*}
On the other hand, $\mathcal{T}_{-1/2}$ satisfies
\begin{align*}
    \operatorname{tr}\mathcal T_f=
    \frac12
    \operatorname{tr}
    \bigl((\operatorname{tr}\Theta_f)I-\Theta_f\bigr)=
    \operatorname{tr}\Theta_f.
\end{align*}
Therefore,
\begin{align*}
    \int_{\mathbb R^3}|v|^2\mathcal M_{-1/2}[f](v)\,dv
    =
    \rho_f\bigl(|U_f|^2+\operatorname{tr}\Theta_f\bigr)
    =
    \int_{\mathbb R^3}|v|^2f(v)\,dv.
\end{align*}
Thus, we conclude
\begin{align*}
    \int_{\mathbb R^3}
   \begin{pmatrix}
				1\\ v\\ |v|^2
			\end{pmatrix}\mathcal M_{-1/2}[f](t,x,v)\,dv
    =
    \int_{\mathbb R^3}
     \begin{pmatrix}
				1\\ v\\ |v|^2
			\end{pmatrix}f(t,x,v)\,dv
\end{align*}
for every $t\in[0,T_*]$ and almost every $x$. Integrating \eqref{eq:endpoint-mild-form-final} and using the
moment identities above, we obtain
\begin{align*}
    \int_{\mathbb{R}^6}\begin{pmatrix}
				1\\ v\\ |v|^2
			\end{pmatrix}f(t,x,v)dxdv&=e^{-t}\int_{\mathbb{R}^6}\begin{pmatrix}
				1\\ v\\ |v|^2
			\end{pmatrix}f_0(x-vt,v)dxdv\cr
    &+\int_0^t\int_{\mathbb{R}^6} e^{-(t-s)}\begin{pmatrix}
				1\\ v\\ |v|^2
			\end{pmatrix}\mathcal{M}_{-1/2}[f](s,x-v(t-s),v)\,dxdvds\cr
    &=e^{-t}\int_{\mathbb{R}^6}\begin{pmatrix}
				1\\ v\\ |v|^2
			\end{pmatrix}f_0(x,v)dxdv\cr
    &+\int_0^t \int_{\mathbb{R}^6} e^{-(t-s)}\begin{pmatrix}
				1\\ v\\ |v|^2
			\end{pmatrix}f(s,x,v)\,dxdvds.
\end{align*}
Then, Gronwall's lemma proves
\eqref{eq:endpoint-conservation-laws}.

Finally, we prove the entropy dissipation
\eqref{eq:endpoint-entropy-dissipation}. Set
\begin{align*}
    \begin{split}
        \mathcal{H}(h):=\int_{\mathbb R^6}h\ln h\,dx\,dv.   
    \end{split}
\end{align*}
Then, the weighted bounds ensure absolute integrability. Let $\mathcal{G}_f$ be
\begin{align*}
    \begin{split}
        \mathcal{G}_f:=\frac{\rho_f}{\sqrt{\det (2\pi\Theta_f)}}\exp\left(-\frac{1}{2}(v-U_f)^{T}\Theta^{-1}_f (v-U_f) \right).
    \end{split}
\end{align*}
Since $f$ and $\mathcal{G}_f$ share the same moments, we have
\begin{align*}
    \begin{split}
        \int_{\mathbb{R}^3}f\ln\mathcal{G}_fdv=\int_{\mathbb{R}^3}\mathcal{G}_f\ln\mathcal{G}_fdv.
    \end{split}
\end{align*}
This gives 
\begin{align}\label{eq:entropy compare 1}
    \begin{split}
        \mathcal{H}(f)-\mathcal{H}(\mathcal{G}_f)&=\int_{\mathbb{R}^3}\left(\int_{\mathbb{R}^3} f\ln f dv-\int \mathcal{G}_f\ln\mathcal{G}_f dv\right)dx\cr
        &=\int_{\mathbb{R}^3}\int_{\mathbb{R}^3} f\ln\frac{f}{\mathcal{G}_f}dvdx\cr
        &=\int_{\mathbb{R}^3}\left(\int_{\mathbb{R}^3} \mathcal{G}_f z\ln z dv \right)dx\cr
        &=\int_{\mathbb{R}^3}\left(\int_{\mathbb{R}^3} \mathcal{G}_f (z\ln z-z+1) dv \right)dx\cr
        &\geq 0,
    \end{split}
\end{align}
where we used $z=f/\mathcal{G}_f$ and $z\ln z-z+1\geq 0$. On the other hand, we have
\begin{align}\label{eq:entropy compare 2}
    \begin{split}
        \mathcal{H}(\mathcal{M}_{-1/2}[f])-\mathcal{H}(\mathcal{G}_f)=-\int_{\mathbb{R}^3}\frac{\rho_f}{2}\ln\frac{\det \mathcal{T}_{-1/2}[f]}{\det \Theta_f}dx\leq 0,
    \end{split}
\end{align}
via $\det\mathcal T_{-1/2}[f]\ge\det\Theta_f$. Therefore, \eqref{eq:entropy compare 1} and \eqref{eq:entropy compare 2} yield $\mathcal{H}(\mathcal M_{-1/2}[f])\le \mathcal{H}(f)$. Applying Jensen's inequality to \eqref{eq:endpoint-mild-form-final} and using the entropy comparison above, we obtain  
\begin{align*}
    \begin{split}
        \mathcal{H}(f(t))&\leq e^{-t}\mathcal{H}(f_0)+\int^t_0 e^{-(t-s)}\mathcal{H}(\mathcal{M}_{-1/2}[f](s))ds\cr
        &\leq e^{-t}\mathcal{H}(f_0)+\int^t_0 e^{-(t-s)}\mathcal{H}(f(s))ds.
    \end{split}
\end{align*} 
Then Gronwall's lemma leads to \eqref{eq:endpoint-entropy-dissipation}. This completes the proof.

\section{Global weak solution}\label{sec:endpoint-weak-convergence}
Global weak solutions to the ES-BGK model at the critical value $\nu=-1/2$ have already been obtained in \cite{son2022cauchy}. Here we also derive a global weak existence result by using the mild solutions constructed in Section~\ref{sec:endpoint-mild-solutions} as approximate solutions. Throughout this section, we define $\mathcal{M}_{-1/2}[f]=0$ on the vacuum set $\{\rho_f=0 \}$.
\begin{definition}
    For $T_*>0$, we say that $f\in L^1([0,T_*]\times\mathbb{R}^3\times\mathbb{R}^3)$ is a weak solution to \eqref{eq:endpoint-esbgk-eq} if $f\ge 0$ satisfies  
    \begin{align*}
        \begin{split}
            &-\int_{\mathbb{R}^3\times\mathbb{R}^3} f_0\phi_0dxdv - \int^{T_*}_0\int_{\mathbb{R}^3 \times\mathbb{R}^3}f(\partial_t +v\cdot \nabla_x)\phi dxdvdt\cr
            &=\int^{T_*}_0\int_{\mathbb{R}^3 \times\mathbb{R}^3}(\mathcal{M}_{-1/2}[f]-f)\phi dxdvdt,
        \end{split}
    \end{align*}
    where $\phi\in C^1_c(\mathbb{R}_+\times\mathbb{R}^3\times\mathbb{R}^3)$ is a test function with $\phi(T_*,x,v)=0$. 
\end{definition}

\begin{theorem}\label{thm:weak}
    Let $\nu=-1/2$. Assume that $f_0(x,v)\geq 0$ satisfies
    \begin{align*}
        \begin{split}
            \int_{\mathbb{R}^6}(1+|v|^2+|x|^2+|\ln f_0|)f_0 dxdv<\infty.
        \end{split}
    \end{align*}
    Then, there exists a global weak solution $f(t,x,v)\geq 0$ to \eqref{eq:endpoint-esbgk-eq} which satisfies
    \begin{align*}
        \begin{split}
            \int_{\mathbb{R}^6}(1+|v|^2+|x|^2+|\ln f(t)|)f(t)dxdv<\infty,
        \end{split}
    \end{align*}
    with the conservation laws and the entropy dissipation:
    \begin{align*}
        \begin{split}
            &\int_{\mathbb{R}^6}\begin{pmatrix}
                1\\ v\\ |v|^2
            \end{pmatrix} f(t) dxdv=  \int_{\mathbb{R}^6}\begin{pmatrix}
                1\\ v\\ |v|^2
            \end{pmatrix} f_0 dxdv,\cr
            &\int_{\mathbb{R}^6}f(t)\ln f(t)dxdv\leq \int_{\mathbb{R}^6}f(s)\ln f(s)dxdv,\text{ for }t\geq s.
        \end{split}
    \end{align*}
\end{theorem}
    
Following the compactness strategy of \cite{park2016cauchy,perthame1989global}, we prove Theorem~\ref{thm:weak} as following process:
    \begin{enumerate}
        \item Construct the approximate solution $f^\varepsilon$ satisfying uniform mass, energy, spatial second moment, and entropy bounds, together with uniform local third velocity moment bounds (see Lemma~\ref{lem:wellposedness-f-eps-data}).
        
        \item Extract a weakly convergent subsequence $f^\varepsilon\rightharpoonup f$ in $L^1$ using the Dunford--Pettis criterion (see Lemma~\ref{lem:endpoint-weak-compactness-f}).
        
        \item Establish strong local $L^1$ convergence of the macroscopic fields using velocity averaging and moment bounds (see Lemma~\ref{lem:endpoint-strong-macro-convergence}).
        
        \item Prove relative weak compactness of $\mathcal{M}_{-1/2}[f^{\varepsilon}]$ in $L^1$ and identify its limit as $\mathcal{M}_{-1/2}[f]$ (see Lemma~\ref{lem:endpoint-maxwellian-uniform-integrability} and Proposition~\ref{prop:endpoint-identification-M}).
    \end{enumerate}
    After this process, since the Gaussian is the only nonlinear term in the ellipsoidal BGK model, the weak convergence $\mathcal{M}_{-1/2}[f^{\varepsilon}]\rightharpoonup \mathcal{M}_{-1/2}[f]$ allows us to pass to the limit in the weak formulation and conclude that $f$ is a weak solution.
    
    We assume that there is initial data $f_0$ that satisfies
    \begin{align}\label{eq:sec 5 f initial assumption}
        \begin{split}
            f_0\geq0,\qquad \int_{\mathbb{R}^6}(1+|x|^2+|v|^2+|\ln f_0|)f_0 dxdv<\infty.    
        \end{split}
    \end{align}
    Let $\zeta\in C^{\infty}_c(\mathbb{R}^6)$ be nonnegative, supported in the unit ball, and satisfy
    \begin{align*}
        \begin{split}
            \int_{\mathbb{R}^6}\zeta dxdv=1.
        \end{split}
    \end{align*}
    Set 
    \begin{align*}
        \begin{split}
            \zeta_\varepsilon(x,v):=\varepsilon^{-6}\zeta\left(\frac{x}{\varepsilon},\frac{v}{\varepsilon} \right).   
        \end{split}
    \end{align*}
    Then, for $r>5$ and $q>5$, we define the regularized initial data $f^\varepsilon_0$ as
    \begin{align}\label{eq:f epsilon define}
        \begin{split}
            f^\varepsilon_0=\zeta_{\varepsilon}*\bigl(f_0 \chi_{\{|x|^2+|v|^2\le \varepsilon^{-2} \} }\bigr) +\varepsilon\langle x\rangle^{-r}e^{-|v|^2},\qquad 0<\varepsilon\leq 1,    
        \end{split}
    \end{align}
    We can easily see that
    \begin{align*}
        \begin{split}
            f^\varepsilon_0\geq0,\qquad \sup_{0<\varepsilon\leq 1}\int_{\mathbb{R}^6}(1+|x|^2+|v|^2+|\ln f^\varepsilon_0|)f^\varepsilon_0 dxdv<\infty,    
        \end{split}
    \end{align*}
    and
    \begin{align*}
        \begin{split}
            f^{\varepsilon}_0\longrightarrow f_0,\text{ in }L^1_2(\mathbb{R}^3\times\mathbb{R}^3).
        \end{split}
    \end{align*}
    \begin{remark}
        In Section~\ref{sec:endpoint-mild-solutions}, the condition $r>3$ suffices for spatial integrability. Here we choose $r>5$ to ensure that the positive background $\varepsilon\langle x\rangle^{-r}e^{-|v|^2}$ in \eqref{eq:f epsilon define} has a finite second spatial moment. We note that the exponent $r$ is a parameter of the regularization and imposes no additional decay assumption on $f_0$.
    \end{remark}
     Now we consider a family of nonnegative approximate
	solutions \(\{f^\varepsilon\}_{\varepsilon>0}\) of
	\begin{align}
		\partial_t f^\varepsilon+v\cdot\nabla_x f^\varepsilon
		=
		\mathcal M_{-1/2}[f^\varepsilon]-f^\varepsilon,
		\qquad
		f^\varepsilon\big|_{t=0}=f_0^\varepsilon,
		\label{eq:endpoint-approximate-equation}
	\end{align}
	where
	\begin{align*}
		\mathcal M_{-1/2}[f^\varepsilon](t,x,v)
		=
		\frac{\rho^\varepsilon}
		{\sqrt{\det(2\pi\mathcal T^\varepsilon)}}
		\exp\left(
		-\frac12
		(v-U^\varepsilon)^T(\mathcal T^\varepsilon)^{-1}
		(v-U^\varepsilon)
		\right),
	\end{align*}
	and
	\begin{align*}
		\mathcal T^\varepsilon
		:=
		\mathcal T_{-1/2}[f^\varepsilon]
		=
		\frac12\bigl((\operatorname{tr}\Theta^\varepsilon)I-\Theta^\varepsilon\bigr).
	\end{align*}
	As before,
	\begin{align*}
		\rho^\varepsilon
		:={{\int_{\mathbb R^3}}f^\varepsilon\,dv},\qquad
        \rho^\varepsilon U^\varepsilon
		:={{\int_{\mathbb R^3}}v f^\varepsilon\,dv},\qquad
        \rho^\varepsilon\Theta^\varepsilon
		&:={{\int_{\mathbb R^3}}
			(v-U^\varepsilon)\otimes(v-U^\varepsilon)f^\varepsilon\,dv}.
	\end{align*}
    
    \begin{lemma}\label{lem:wellposedness-f-eps-data}
        Let $f_0$ and $f^{\varepsilon}_0$ be defined in \eqref{eq:sec 5 f initial assumption} and \eqref{eq:f epsilon define}. Then, for each $0<\varepsilon\leq 1$ and final time $T_*>0$, there exists a unique nonnegative mild solution $f^\varepsilon\in C\bigl([0,T_*];L^1_2(\mathbb{R}^6) \bigr)$ to \eqref{eq:endpoint-approximate-equation}. Moreover, the following bounds hold uniformly in $\varepsilon$:
        \begin{align}
		\int_0^{T_*}\!\int_{\mathbb R^3_x\times\mathbb R^3_v}
		\bigl(1+|v|^2+|x|^2+|\ln f^\varepsilon|\bigr)
		f^\varepsilon\,dv\,dx\,dt
		\le C_{T_*},
		\label{eq:uniform-f-eps-entropy-moment}
	\end{align}
    and for every compact set $K_x\subset\subset\mathbb{R}^3_x$,
    \begin{align}
		\int_0^{T_*}\!\int_{K_x}\!\int_{\mathbb R^3_v}
		|v|^3 f^\varepsilon\,dv\,dx\,dt
		\le
		C_{T_*,K_x}.
		\label{eq:uniform-local-third-moment}
	\end{align}
    \end{lemma}
    \begin{proof}
        For each fixed $\varepsilon>0$, the regularized initial datum $f_0^\varepsilon$ satisfies the assumptions of Theorem~\ref{thm:endpoint-mild-solution}, which yields the existence and uniqueness of the corresponding mild solution. To obtain estimates independent of $\varepsilon$, we use the uniform initial moment and entropy bounds:
        \begin{enumerate}
            \item Mass, energy, and spatial second-moment estimate:
            \begin{align*}
                \begin{split}
                    \int_{\mathbb{R}^6}(1+|v|^2)f^\varepsilon(t)dxdv=\int_{\mathbb{R}^6}(1+|v|^2)f^\varepsilon_0dxdv\leq C,
                \end{split}
            \end{align*}
            and
            \begin{align*}
                \begin{split}
                    \frac{d}{dt}\int_{\mathbb{R}^6}|x|^2f^{\varepsilon}dxdv=2\int_{\mathbb{R}^6}x\cdot vf^{\varepsilon}dxdv\leq \int_{\mathbb{R}^6}(|x|^2+|v|^2)f^{\varepsilon}dxdv
                \end{split}
            \end{align*}
            \item Entropy bound by $H$-theorem:
            \begin{align*}
                \begin{split}
                    \int_{\mathbb{R}^6}f^{\varepsilon}|\ln f^{\varepsilon}|dxdv=\mathcal{H}(f^\varepsilon)+2\int_{ \{f^\varepsilon\leq 1\} }-f^{\varepsilon}\ln f^{\varepsilon}dxdv\leq C_{T_*}.
                \end{split}
            \end{align*}
          \end{enumerate}
          Finally, we can also prove \eqref{eq:uniform-local-third-moment} by referring to Lemma 2 in \cite{perthame1989global}.
    \end{proof}
	
	\begin{lemma}
		\label{lem:endpoint-weak-compactness-f}
		There exists a subsequence, still denoted by \(f^\varepsilon\), and a
		nonnegative function
		\begin{align*}
			f\in L^1\bigl((0,T_*)\times\mathbb R^3_x\times\mathbb R^3_v\bigr)
		\end{align*}
		such that
		\begin{align}
			f^\varepsilon\rightharpoonup f
			\quad\text{weakly in }L^1,
			\label{eq:f-eps-weak-L1}
		\end{align}
		and
		\begin{align}
			vf^\varepsilon\rightharpoonup vf
			\quad\text{weakly in }L^1.
			\label{eq:vf-eps-weak-L1}
		\end{align}
		Moreover, for every compact set \(K_x\subset\subset\mathbb R^3_x\),
		\begin{align}
			v\otimes v\, f^\varepsilon
			\rightharpoonup
			v\otimes v\, f
			\quad\text{weakly in }
			L^1\bigl((0,T_*)\times K_x\times\mathbb R^3_v\bigr).
			\label{eq:vvf-eps-weak-local}
		\end{align}
	\end{lemma}
	
	\begin{proof}
		The bound \eqref{eq:uniform-f-eps-entropy-moment} gives uniform integrability and tightness of $\{f^{\varepsilon}\}$ and $\{vf^{\varepsilon} \}$. Hence, following the Dunford--Pettis criterion, we deduce
		\eqref{eq:f-eps-weak-L1} and \eqref{eq:vf-eps-weak-L1}, after extracting a
		subsequence.
		
		Fix $K_x\subset\subset\mathbb{R}^3_x$. The local third
		moment bound \eqref{eq:uniform-local-third-moment} implies uniform
		integrability of $\{v\otimes v f^\varepsilon\}$ on
		$(0,T_*)\times K_x\times\mathbb R^3_v$.  Indeed, for $R>1$,
		\begin{align*}
			\int_0^{T_*}\!\int_{K_x}\!\int_{|v|>R}
			|v|^2 f^\varepsilon\,dv\,dx\,dt
			\le
			\frac1R
			\int_0^{T_*}\!\int_{K_x}\!\int_{\mathbb R^3}
			|v|^3 f^\varepsilon\,dv\,dx\,dt
			\le
			\frac{C_{{T_*},K_x}}R.
		\end{align*}
		On $|v|\le R$, uniform integrability follows from that of
		$f^\varepsilon$. Therefore, the Dunford--Pettis
		theorem yields, up to a subsequence, \eqref{eq:vvf-eps-weak-local}, with the limit identified by \eqref{eq:f-eps-weak-L1}.
	\end{proof}
	
	\begin{lemma}[Strong convergence of macroscopic fields]
		\label{lem:endpoint-strong-macro-convergence}
		Let
		\begin{align*}
			\rho:=\int_{\mathbb R^3}f\,dv,
			\qquad
			m:=\int_{\mathbb R^3}vf\,dv,
			\qquad
			S:=\int_{\mathbb R^3}v\otimes v f\,dv.
		\end{align*}
		Then, after extracting a further subsequence if necessary, for every compact
		set \(K_x\subset\subset\mathbb R^3_x\),
		\begin{align}
			\rho^\varepsilon\to\rho,
			\qquad
			\rho^\varepsilon U^\varepsilon\to m,
			\qquad
			S^\varepsilon:=\int_{\mathbb R^3}v\otimes v f^\varepsilon\,dv\to S
			\label{eq:macro-strong-L1}
		\end{align}
		strongly in \(L^1((0,T_*)\times K_x)\) and almost everywhere on
		\((0,T_*)\times\mathbb R^3_x\).  In particular, on the set
		\begin{align*}
			E:=\{(t,x):\rho(t,x)>0\},
		\end{align*}
		we have
		\begin{align}
			U^\varepsilon\to U:=\frac{m}{\rho},
			\qquad
			\Theta^\varepsilon\to
			\Theta:=\frac{S}{\rho}-U\otimes U,
			\qquad
			\mathcal T^\varepsilon\to
			\mathcal T:=\frac12\bigl((\operatorname{tr}\Theta)I-\Theta\bigr)
			\label{eq:UT-convergence-ae}
		\end{align}
		almost everywhere on \(E\).
	\end{lemma}
	
	\begin{proof}
		By Lemma~\ref{lem:endpoint-weak-compactness-f}, we have already shown the weak compactness of $f^{\varepsilon}$. Thus, the velocity averaging lemma \cite{golse1988regularity,perthame1989global} applied to
		\eqref{eq:endpoint-approximate-equation} gives compactness of the velocity
		averages of \(f^\varepsilon\). Then, we obtain strong local \(L^1\)
		convergence of the averages corresponding to
		\(1\), \(v\), and \(v\otimes v\).  This proves
		\eqref{eq:macro-strong-L1}.
		
		The almost everywhere convergence follows by extracting a subsequence from the
		strong local \(L^1\) convergence. On \(E=\{\rho>0\}\), we may divide by
		\(\rho^\varepsilon\to\rho>0\) to obtain
		\begin{align*}
			U^\varepsilon
			=
			\frac{\rho^\varepsilon U^\varepsilon}{\rho^\varepsilon}
			\to
			\frac{m}{\rho}=U.
		\end{align*}
		Since
		\begin{align*}
			\Theta^\varepsilon
			=
			\frac{S^\varepsilon}{\rho^\varepsilon}
			-
			U^\varepsilon\otimes U^\varepsilon,
		\end{align*}
		we also get \(\Theta^\varepsilon\to\Theta\) almost everywhere on \(E\).  The
		convergence of \(\mathcal T^\varepsilon\) follows from the linear relation
		\begin{align*}
			\mathcal T^\varepsilon
			=
			\frac12\bigl((\operatorname{tr}\Theta^\varepsilon)I-
			\Theta^\varepsilon\bigr).
		\end{align*}
	\end{proof}
	
	\begin{lemma}
		\label{lem:endpoint-maxwellian-uniform-integrability}
		There exists a constant \(C_{T_*}\), independent of \(\varepsilon\), such that
		\begin{align}
			\int_0^{T_*}\!\int_{\mathbb R^3_x\times\mathbb R^3_v}
			\bigl(1+|v|^2+|x|^2+|\ln\mathcal M_{-1/2}[f^\varepsilon]|\bigr)
			\mathcal M_{-1/2}[f^\varepsilon] \,dv\,dx\,dt
			\le
			C_{T_*}.
			\label{eq:endpoint-M-uniform-integrability}
		\end{align}
		Consequently, \(\{\mathcal M_{-1/2}[f^\varepsilon]\}\) is relatively weakly compact in
		\(L^1((0,T_*)\times\mathbb R^3_x\times\mathbb R^3_v)\).
	\end{lemma}
	
	\begin{proof}
		We write
		\begin{align*}
			\mathcal M^\varepsilon:=\mathcal M_{-1/2}[f^\varepsilon].
		\end{align*}
		The mass and energy moments of \(\mathcal M^\varepsilon\) coincide with those
		of \(f^\varepsilon\).  Indeed,
		\begin{align*}
			\int_{\mathbb R^3}\mathcal M^\varepsilon\,dv
			=
			\rho^\varepsilon
			=
			\int_{\mathbb R^3}f^\varepsilon\,dv,
		\end{align*}
		and, since
		\(\operatorname{tr}\mathcal T^\varepsilon=\operatorname{tr}\Theta^\varepsilon\),
		\begin{align*}
			\int_{\mathbb R^3}|v|^2\mathcal M^\varepsilon\,dv
			=
			\rho^\varepsilon
			\bigl(|U^\varepsilon|^2+
			\operatorname{tr}\mathcal T^\varepsilon\bigr)
			=
			\rho^\varepsilon
			\bigl(|U^\varepsilon|^2+
			\operatorname{tr}\Theta^\varepsilon\bigr)
			=
			\int_{\mathbb R^3}|v|^2f^\varepsilon\,dv.
		\end{align*}
		Also,
		\begin{align*}
			\int_{\mathbb R^3}|x|^2\mathcal M^\varepsilon\,dv
			=
			|x|^2\rho^\varepsilon
			=
			\int_{\mathbb R^3}|x|^2f^\varepsilon\,dv.
		\end{align*}
		Thus the part of \eqref{eq:endpoint-M-uniform-integrability} without the
		logarithm follows directly from \eqref{eq:uniform-f-eps-entropy-moment}.
		
		We now estimate the logarithmic part. First, we control the negative part of $\mathcal{M}^{\varepsilon}\ln \mathcal{M}^{\varepsilon}$. To do this, we set $R(x,v):=|x|^2+|v|^2$. For $e^{-R}\le \mathcal{M}^\varepsilon\leq 1$, we have
        \begin{align*}
            \begin{split}
                -\mathcal{M}^\varepsilon \ln \mathcal{M}^\varepsilon\leq R\mathcal{M}^\varepsilon.
            \end{split}
        \end{align*}
        For $0\leq \mathcal{M}^\varepsilon<e^{-R}$, the elementary bounds
        \begin{align*}
            \begin{split}
                -z\ln z\leq Ce^{-\frac{R}{2}},\qquad 0\le z\le e^{-R},
            \end{split}
        \end{align*}
        imply 
        \begin{align*}
            \begin{split}
                -\mathcal{M}^\varepsilon \ln \mathcal{M}^\varepsilon\leq Ce^{-\frac{R}{2}}.    
            \end{split}
        \end{align*}
        Therefore, the negative part is bounded as follows:
        \begin{align}\label{eq:Mx-negative}
            \begin{split}
                \int_{ \{\mathcal{M}^\varepsilon\leq 1 \} }-\mathcal{M}^\varepsilon \ln \mathcal{M}^\varepsilon dxdvdt&\leq \int^{T_*}_0\int_{\mathbb{R}^6}(|x|^2+|v|^2)\mathcal{M}^\varepsilon dxdvdt\cr
                &+C\int^{T_*}_0 \int_{\mathbb{R}^6} e^{-\frac{|x|^2+|v|^2}{2}}dxdvdt\cr
                &\leq C_{T_*}.
            \end{split}
        \end{align}
        On the other hand, we recall that
        \begin{align}\label{eq:Mx-total}
            \begin{split}
                \int^{T_*}_0 \int_{\mathbb{R}^6}\mathcal{M}^\varepsilon \ln \mathcal{M}^\varepsilon dxdvdt\leq \int^{T_*}_0 \int_{\mathbb{R}^6}f^\varepsilon \ln f^{\varepsilon}dxdvdt.
            \end{split}
        \end{align}
        This follows from
        \eqref{eq:entropy compare 1} and \eqref{eq:entropy compare 2}, which were used to obtain \eqref{eq:endpoint-entropy-dissipation}.
        Combining \eqref{eq:Mx-total} and \eqref{eq:Mx-negative}, \eqref{eq:uniform-f-eps-entropy-moment} in Lemma~\ref{lem:wellposedness-f-eps-data} leads to
        \begin{align*}
            \begin{split}
                \int^{T_*}_0 \int_{\mathbb{R}^6}\mathcal{M}^{\varepsilon}|\ln \mathcal{M}^\varepsilon|dxdvdt&= \int^{T_*}_0 \int_{\mathbb{R}^6}\mathcal{M}^{\varepsilon}\ln \mathcal{M}^\varepsilon dxdvdt + 2\int_{ \{\mathcal{M}^\varepsilon\leq 1 \} }-\mathcal{M}^\varepsilon \ln \mathcal{M}^\varepsilon dxdv dt\cr
                &\leq \int^{T_*}_0 \int_{\mathbb{R}^6}f^\varepsilon \ln f^{\varepsilon}dxdvdt+  2C_{T_*}\cr
                &\leq C_{T_*}.
            \end{split}
        \end{align*}
        Relative weak compactness follows again from \eqref{eq:endpoint-M-uniform-integrability} via the Dunford--Pettis theorem. This completes the proof.
	\end{proof}
	
	\begin{proposition}
		\label{prop:endpoint-identification-M}
		After extracting a subsequence, we have
		\begin{align*}
			\mathcal M_{-1/2}[f^\varepsilon]
			\rightharpoonup
			\mathcal M_{-1/2}[f]
			\quad\text{weakly in }
			L^1\bigl((0,T_*)\times\mathbb R^3_x\times\mathbb R^3_v\bigr),
		\end{align*}
		where \(\mathcal M_{-1/2}[f]\) is defined to be zero on the vacuum set
		\(\{\rho=0\}\).
	\end{proposition}
	
	\begin{proof}
		By Lemma~\ref{lem:endpoint-maxwellian-uniform-integrability}, there exists a nonnegative function \(M\in L^1\) such that, up to a subsequence,
		\begin{align}
			\mathcal M^\varepsilon:=\mathcal M_{-1/2}[f^\varepsilon]
			\rightharpoonup M
			\quad\text{weakly in }L^1.
			\label{eq:M-eps-weak-to-M}
		\end{align}
		We identify \(M\).  On the non-vacuum set
		\begin{align*}
			E:=\{(t,x):\rho(t,x)>0\},
		\end{align*}
		Lemma~\ref{lem:endpoint-strong-macro-convergence} gives
		\begin{align*}
			\rho^\varepsilon\to\rho,
			\qquad
			U^\varepsilon\to U,
			\qquad
			\mathcal T^\varepsilon\to\mathcal T
		\end{align*}
		almost everywhere.  By Lemma~\ref{lem:theta-positive-definite}, for almost
		every \((t,x)\in E\), the stress tensor \(\Theta(t,x)\) is positive definite.
		Therefore \(\mathcal T(t,x)\) is positive definite, and the Gaussian formula is
		continuous in the parameters. Hence
		\begin{align}
			\mathcal M^\varepsilon(t,x,v)
			\to
			\mathcal M_{-1/2}[f](t,x,v)
			\quad\text{for a.e. }(t,x,v)\in E\times\mathbb R^3_v.
			\label{eq:M-pointwise-on-E}
		\end{align}
		
		Let \(\varphi\in L^\infty\) be nonnegative.  Fatou's lemma and
		\eqref{eq:M-eps-weak-to-M} imply
		\begin{align}
			\int_{E\times\mathbb R^3_v}
			\mathcal M_{-1/2}[f]\varphi\,dv\,dx\,dt
			\le
			\liminf_{\varepsilon\to0}
			\int_{E\times\mathbb R^3_v}
			\mathcal M^\varepsilon\varphi\,dv\,dx\,dt
			=
			\int_{E\times\mathbb R^3_v}M\varphi\,dv\,dx\,dt.
			\label{eq:M-limit-inequality}
		\end{align}
		Thus \(\mathcal M_{-1/2}[f]\le M\) almost everywhere on
		\(E\times\mathbb R^3_v\).  Taking \(\varphi\equiv1\) in
		\eqref{eq:M-limit-inequality}, and using \(\int\mathcal M^\varepsilon\,dv=\rho^\varepsilon\), we obtain
		\begin{align*}
			\int_{E\times\mathbb R^3_v}M\,dv\,dx\,dt
			=
			\lim_{\varepsilon\to0}\int_E\rho^\varepsilon\,dx\,dt
			=
			\int_E\rho\,dx\,dt
			=
			\int_{E\times\mathbb R^3_v}
			\mathcal M_{-1/2}[f] \,dv\,dx\,dt,
		\end{align*}
		where the convergence of \(\rho^\varepsilon\) is strong in
		\(L^1_{\mathrm{loc}}\), and the spatial moment bound allows passage to the
		whole space by truncation. Therefore, we have
		\begin{align*}
			M=
			\mathcal M_{-1/2}[f]
			\quad\text{a.e. on }E\times\mathbb R^3_v.
		\end{align*}
		
		It remains to consider the vacuum set \(E^c\).  Since \(\rho=0\) on \(E^c\)
		and \(\rho^\varepsilon\to\rho\) strongly in \(L^1_{\mathrm{loc}}\), we have
		for every compact set \(K_x\subset\subset\mathbb R^3_x\),
		\begin{align*}
			\int_{E^c\cap((0,T_*)\times K_x)}\!\int_{\mathbb R^3_v}
			\mathcal M^\varepsilon\,dv\,dx\,dt
			=
			\int_{E^c\cap((0,T_*)\times K_x)}\rho^\varepsilon\,dx\,dt
			\to0.
		\end{align*}
		The spatial moment bound again removes the compact restriction.  Hence
		\(M=0\) almost everywhere on \(E^c\times\mathbb R^3_v\).  This completes the proof.
	\end{proof}

\noindent{\bf Acknowledgment} This work was supported by the National Research Foundation of
Korea(NRF) grant funded by the Korea government(MSIT) (No.RS-2023-NR076676).\newline
\newline
\noindent\textbf{Use of Generative-AI tools declaration}
The authors used ChatGPT (OpenAI) to assist with English drafting and revision, and literature searches. This assistance concerned the Introduction and the exposition of proofs in Sections 2--5.

\bibliographystyle{abbrv}
\bibliography{ref} 
\end{document}